\documentclass[11pt]{article}
\usepackage{amsfonts}
\usepackage{amsmath,amssymb}
\usepackage{amsthm}
\usepackage[mathscr]{euscript}
\usepackage{mathrsfs}
\usepackage{indentfirst}
\usepackage{color}\usepackage{enumitem}

\newtheorem{theorem}{Theorem}[section]
\newtheorem{lemma}[theorem]{Lemma}

\newtheorem{remark}[theorem]{Remark}

\numberwithin{equation}{section}

\allowdisplaybreaks

\newcommand{\be}{\begin{equation}}
\newcommand{\ee}{\end{equation}}
\newcommand\bes{\begin{eqnarray}} \newcommand\ees{\end{eqnarray}}
\newcommand{\bess}{\begin{eqnarray*}}
\newcommand{\eess}{\end{eqnarray*}}
\newcommand{\bbbb}{\left\{\begin{aligned}}
\newcommand{\nnnn}{\end{aligned}\right.}
\newcommand{\bea}{\begin{align*}}
\newcommand{\eea}{\end{align*}}
\newtheorem{thm}{Theorem}
\newcommand\ep{\varepsilon}

\newcommand\dd{\displaystyle}

\newcommand\dx{{\rm d}x}
\newcommand\dy{{\rm d}y}
\newcommand\dt{{\rm d}t}

\newcommand\yy{\infty}

\newcommand\sk{\smallskip}

\begin{document}\thispagestyle{empty}
\setlength{\baselineskip}{16pt}

\begin{center}
 {\LARGE\bf On monostable cooperative system with nonlocal diffusion and free boundaries\footnote{This work was supported by NSFC Grants
11771110, 11971128, 11901541}}\\[4mm]
  {\Large Lei Li\textsuperscript{\dag}, \ \  Xueping Li\textsuperscript{\ddag}, \ \ Mingxin Wang\textsuperscript{\dag}\footnote{Corresponding author. {\sl E-mail}: mxwang@hit.edu.cn}}\\[1.5mm]
\textsuperscript{\dag}{School of Mathematics, Harbin Institute of Technology, Harbin 150001, PR China}\\
\textsuperscript{\ddag}{School of Mathematics and Information Science, Zhengzhou University of Light Industry, Zhengzhou 450002, PR China}
\end{center}

\date{\today}

\begin{abstract}
This paper continues to study the monostable cooperative system with nonlocal diffusion and free boundary, which has recently been discussed by [Du and Ni, 2020, arXiv:2010.01244]. We here aim at the four aspects: the first is to give more accurate estimates for solution; the second is to discuss the limitations of solution pair of a semi-wave problem; the third is to investigate the asymptotic behaviors of the corresponding cauchy problem; the last is to study the limiting profiles of solution as one of the expanding rates of free boundary converges to $\yy$.

\textbf{Keywords}: Nonlocal diffusion; free boundary; accelerated spreading; spreading speed.

\textbf{AMS Subject Classification (2000)}: 35K57, 35R09,
35R20, 35R35, 92D25

\end{abstract}

\section{Introduction}
\renewcommand{\thethm}{\Alph{thm}}
Recently, Du and Ni \cite{DN21} considered the following monostable cooperative systems with nonlocal diffusion and free boundaries
 \bes\label{1.1}\left\{\begin{aligned}
&\partial_t u_i=d_i\mathcal{L}[u_i](t,x)+f_i(u_1,u_2,\cdots,u_m),&&t>0,~x\in(g(t),h(t)),1\le i\le m_0,\\
&\partial_t u_i=f_i(u_1,u_2,\cdots,u_m), &&t>0,~x\in(g(t),h(t)),~m_0\le i\le m,\\
&u_i(t,g(t))=u_i(t,h(t))=0,& &t>0, ~1\le i\le m,\\
&g'(t)=-\sum_{i=1}^{m_0}\mu_i\int_{g(t)}^{h(t)}\int_{-\yy}^{g(t)}J_i(x-y)u_i(t,x){\rm d}y{\rm d}x, & & t>0,\\
&h'(t)=\sum_{i=1}^{m_0}\mu_i\int_{g(t)}^{h(t)}\int_{h(t)}^{\yy}J_i(x-y)u_i(t,x){\rm d}y{\rm d}x, & & t>0,\\
&h(0)=-g(0)=h_0>0,\;\; u_i(0,x)=u_{i0}(x),& &|x|\le h_0, ~1\le i\le m,
 \end{aligned}\right.
 \ees
where $1\le m_0\le m$, $d_i>0$, $\mu_i\ge0$, $\sum_{i=1}^{m_0}\mu_i>0$, and
 \bes\label{1.dn}
 \mathcal{L}[u_i](t,x):=\int_{g(t)}^{h(t)}J_i(x-y)u_i(t,y){\rm d}y-u_i(t,x).
\ees
For $1\le i\le m_0$, kernel functions $J_i$ satisfy
\begin{enumerate}[leftmargin=4em]
\item[{\bf(J)}]$J\in C(\mathbb{R})\cap L^{\yy}(\mathbb{R})$, $J\ge 0$, $J(0)>0,~\dd\int_{\mathbb{R}}J(x)\dx=1$, \ $J$\; is even,
 \end{enumerate}
 and initial functions $u_{i0}(x)$ satisfy
 \[u_{i0}\in C([-h_0,h_0]), ~ u_{i0}(\pm h_0)=0<u_{i0}(x), ~ ~ \forall ~ x\in(-h_0,h_0).\]

 This model can be used to describe the spreading of some epidemics and the interactions of various species, for example, see \cite{ZZLD} and \cite{DNwn}, where the spatial movements of agents are approximated by the nonlocal diffusion operator \eqref{1.dn} instead of random diffusion (also known as local diffusion). Such kind of free boundary problem was firstly proposed in \cite{CDLL} and \cite{CQW}. Especially, it can be seen from \cite{CDLL} that the introduction of nonlocal diffusion brings about some different dynamical behaviors from the local version in \cite{DL2010}, and also gives arise to some technical difficulties. Since these two works appeared, some related research has emerged. For example, one can refer to \cite{DLZ} for the first attempt to the spreading speed of \cite{CDLL}, \cite{LWW20,DWZ,CLWZ} for the Lotka-Volterra competition and prey-predator models, \cite{DN212} for high dimensional and radial symmetric version for Fisher-KPP equation, and \cite{LW21} for the model with a fixed boundary and a moving boundary.

 Before introducing our results for \eqref{1.1}, let us briefly review some conclusions obtained by Du and Ni \cite{DN21}. The following notations and assumptions are necessary.

{\bf Notations:}

(i)\ $\mathbb{R}^m_+:=\{x=(x_1,x_2,\cdots,x_m)\in\mathbb{R}^m:x_i\ge0, ~ 1\le i\le m\}$.

(ii)\  For any $x=(x_1,x_2,\cdots,x_m)\in\mathbb{R}^m$, simply write $x=(x_i)$. For $x=(x_i),\ y=(y_i)\in\mathbb{R}^m$,
\bess
&&x\preceq(\succeq) y ~ ~ \text{means} ~ ~ x_i\le(\ge) y_i ~ \text{for} ~ 1\le i\le m,\\
&&x\precneqq(\succneqq) y ~ ~ \text{means} ~ ~ x\preceq(\succeq) y ~ \text{but} ~ x\neq y,\\
&&x\prec(\succ) y ~ ~ \text{means} ~ ~ x_i<(>) y_i ~ \text{for} ~ 1\le i\le m.
\eess

(iii)\ If $x\preceq y$, $[x,y]$ represents the set of $\{z\in\mathbb{R}^m: x\preceq z\preceq y\}$.

(iv)\ Hadamard product: For any $x=(x_i),\ y=(y_i)\in\mathbb{R}^m$, $x\circ y:=(x_iy_i)\in\mathbb{R}^m$.

{\bf Assumptions on reaction term $f_i$ :}\\
{\bf(f1)}\, (i)\ Let $F=(f_1,f_2,\cdots, f_m)\in [C^1(\mathbb{R}^m_+)]^m$. $F(u)={\bf 0}$ has only two roots in $\mathbb{R}^m_+$: ${\bf 0}=(0,0,\cdots,0)$ and ${\bf u^*}=(u^*_1,u^*_2,\cdots, u^*_m)\succ {\bf 0}$.

      (ii)\, $\partial_jf_i(u)\ge0$ for $i\neq j$ and $u\in[{\bf 0},{\bf \hat{u}}]$, where either ${\bf \hat{u}}=\yy$ meaning $[{\bf 0},{\bf \hat{u}}]=\mathbb{R}^m_+$, or ${\bf u^*}\prec {\bf \hat{u}}\in\mathbb{R}^m_+$, which implies that system \eqref{1.1} is cooperative in $[{\bf 0},{\bf \hat{u}}]$.

      (iii)\, The matrix $\nabla F({\bf0})=(\partial_jf_i({\bf 0}))_{m\times m}$ is irreducible with positive principal eigenvalue.

      (iv)\, If $m_0<m$, then $\partial_jf_i(u)>0$ for $1\le j\le m_0<i\le m$ and $u\in[{\bf 0},{\bf u^*}]$.\\
 {\bf(f2)}\, $F(ku)\ge kF(u)$ for any $0\le k\le 1$ and $u\in[{\bf 0},{\bf \hat{u}}]$.\\
{\bf(f3)}\, The matrix $\nabla F({\bf u^*})$ is invertible, ${\bf u^*}\nabla F({\bf u^*})\preceq {\bf 0}$ and for every $i\in\{1,2,\cdots, m\}$, either

       (i)\, $\sum_{j=1}^{m}\partial_jf_i({\bf u^*})u^*_j<0$, or

       (ii)\, $\sum_{j=1}^{m}\partial_jf_i({\bf u^*})u^*_j=0$ and $f_i(u)$ is linear in $[{\bf u^*}-\ep_0{\bf 1},{\bf u^*}]$ for some small $\ep_0>0$, where ${\bf 1}=(1,1,\cdots,1)\in\mathbb{R}^m$.\\
{\bf(f4)}\, The ODE system
\[\bar{u}_t=F(\bar{u}), ~ ~ \bar{u}(0)\succ{\bf 0}\]
has a unique global solution $\bar{u}$ and $\lim_{t\to\yy}\bar{u}(t)={\bf u^*}$.\\
{\bf(f5)}\, The problem
\bes\label{1.2}
U_t=D\circ\int_{\mathbb{R}}{\bf J}(x-y)\circ U(t,y)\dy-D\circ U+F(U) ~ ~ \text{for} ~ t>0, ~ x\in\mathbb{R}
\ees
has a invariant set $[{\bf 0},{\bf \hat{u}}]$ and its every nontrivial solution is attracted by the equilibrium ${\bf u^*}$. That is, if the initial $U(0,x)\in[{\bf 0},{\bf \hat{u}}]$, then $U(t,x)\in[{\bf 0},{\bf \hat{u}}]$ for all $t>0$ and $x\in\mathbb{R}$; if further $U(0,x)\not\equiv{\bf0}$, then $\lim_{t\to\yy}U(t,x)={\bf u^*}$ locally uniformly in $\mathbb{R}$.

Throughout this paper, we always make the following assumptions:

(i)\, $J_i$ satisfy condition {\bf (J)} for $i\in\{1,2,\cdots,m_0\}$.

(ii)\, $d_i=0$ and $J_i(x)\equiv0$ for $x\in\mathbb{R}$, $i\in\{m_0+1,m_0+2,\cdots, m\}$, $D=(d_i)$ and ${\bf J}=(J_i(x))$.

(iii)\, {\bf(f1)}-{\bf(f5)} hold true.

(iv)\, initial value $(u_{i0}(x))\in[{\bf 0},{\bf \hat{u}}]$.

Under the above assumptions, one easily proves that \eqref{1.1} has a unique global solution $(u,g,h)$. Here we suppose that its longtime behaviors are governed by a spreading-vanishing dichotomy, namely, one of the following alternatives must happen for \eqref{1.1}

\sk{\rm(i)}\, \underline{Spreading:} $\lim_{t\to\yy}-g(t)=\lim_{t\to\yy}h(t)=\yy$ and $\lim_{t\to\yy}u(t,x)={\bf u^*}$ locally uniformly in $\mathbb{R}$.

\sk{\rm(ii)}\, \underline{Vanishing:} $\lim_{t\to\yy}h(t)-g(t)<\yy$ and $\lim_{t\to\yy}\sum_{i=1}^{m}\|u_i(t,\cdot)\|_{C([g(t),h(t)])}=0$.

The authors of \cite{DN21} discussed the spreading speeds of $g(t)$ and $h(t)$ when spreading happens for \eqref{1.1}, and proved that there is a finite spreading speed for \eqref{1.1} if and only if the following condition holds for $J_i$
\begin{enumerate}[leftmargin=4em]
\item[{\bf(J1)}]$\dd\int_{0}^{\yy}xJ_i(x)\dx<\yy$ for any $i\in\{1,2,\cdots, m_0\}$ such that $\mu_i>0$.
 \end{enumerate}
Exactly, they obtained the following conclusion.
\begin{thm}{\rm \cite[Theorem 1.3]{DN21} }\label{A} Let $(u,g,h)$ be a solution of \eqref{1.1}, and spreading happen. Then
\begin{align*}
\lim_{t\to\yy}\frac{-g(t)}{t}=\lim_{t\to\yy}\frac{h(t)}{t}=\left\{\begin{aligned}
&c_0& &{\rm if~{\bf(J1)}~holds},\\
&\yy& &{\rm if~{\bf(J1)}~does ~ not ~ hold},
\end{aligned}\right.
  \end{align*}
   where $c_0$ is uniquely determined by the semi-wave problem
  \bes\label{1.3}\left\{\begin{array}{lll}
 D\circ\dd\int_{-\yy}^{0}{\bf J}(x-y)\circ \Phi(y)\dy-D\circ \Phi+c\Phi'(x)+F(\Phi)=0, ~ ~ -\yy<x<0,\\
 \Phi(-\yy)={\bf u^*}, ~ ~ \Phi(0)={\bf0}, ~ ~ \Phi(x)=(\phi_i(x)),
 \end{array}\right.
 \ees
 and
 \bes\label{1.4}
 c=\sum_{i=1}^{m_0}\mu_i\int_{-\yy}^{0}\int_{0}^{\yy}J_i(x-y)\phi_i(x)\dy\dx.
 \ees
\end{thm}
When {\bf(J1)} does not hold, we usually call the phenomenon by the accelerated spreading.
For problem \eqref{1.3}, they obtained a dichotomy result between it and the traveling wave problem
\bes\label{1.5}\left\{\begin{array}{lll}
 D\circ\dd\int_{\mathbb{R}}{\bf J}(x-y)\circ \Psi(y)\dy-D\circ \Psi+c\Psi'(x)+F(\Psi)=0, ~ ~ -\yy<x<\yy,\\
 \Psi(-\yy)={\bf u^*}, ~ ~ \Psi(\yy)={\bf0}, ~ ~ \Psi(x)=(\psi_i(x)).
 \end{array}\right.
 \ees
 To state the dichotomy, we firstly show a new condition on $J_i$, namely,
 \begin{enumerate}[leftmargin=4em]
\item[{\bf(J2)}]$\dd\int_{0}^{\yy}e^{\lambda x}J_i(x)\dx<\yy$ for some $\lambda>0$ and any $i\in\{1,2,\cdots, m_0\}$.
 \end{enumerate}
 Clearly, condition {\bf (J2)} indicates {\bf (J1)} but not the other way around.
 \begin{thm}{\rm \cite[Theorem 1.1 and Theorem 1.2]{DN21} }\label{B} The followings hold:

 {\rm(i)} There exists a $C_*\in(0,\yy]$ such that problem \eqref{1.3} has a unique monotone solution if and only if $c\in(0,C_*)$, and \eqref{1.5} has a monotone solution if and only if $c\ge C_*$.

 {\rm(ii)} $C_*\neq\yy$ if and only if {\bf(J2)} holds.

  {\rm(iii)} The semi-wave problem \eqref{1.3}-\eqref{1.4} has a unique solution pair $(c_0,\Phi_0)$ with $c_0>0$ and $\Phi_0$ nonincreasing in $(0,\yy]$ if and only if {\bf(J1)} holds.
 \end{thm}
Some more accurate estimates on free boundaries were also derived if $J_i$ additionally satisfy
\begin{enumerate}[leftmargin=4em]
\item[${\bf(J^\gamma)}$] there exist $C_1,C_2>0$ such that $C_1|x|^{-\gamma}\le J_i(x)\le C_2|x|^{-\gamma}$ for $|x|\gg 1$ and $i\in\{1,\cdots, m_0\}$.
 \end{enumerate}
 \begin{thm}{\rm \cite[Theorem 1.5]{DN21}} Suppose that ${\bf(J^\gamma)}$ holds with $\gamma\in(1,2]$, then
  \bess\left\{\begin{array}{lll}
  -g(t), ~ h(t)\approx t\ln t ~ ~ &&{\rm if} ~ \gamma=2,\\
  -g(t), ~ h(t)\approx t^{1/(\gamma-1)} ~ ~ &&{\rm if} ~ \gamma\in(1,2).
 \end{array}\right.
 \eess
 \end{thm}
Inspired by the above interesting results, we here focus on the following four aspects:

(i) When spreading happens for \eqref{1.1}, we give more accurate longtime behaviors of solution component $u$ rather than that of spreading case mentioned above. Particularly, if ${\bf(J^\gamma)}$ holds with $\gamma\in(1,2]$ and $m_0=m$, then we obtain some sharp estimates on solution component $u$ which are closely related to the behaviors of free boundaries near infinity.

(ii) Assume that {\bf(J1)} holds. Choose a $\mu_i>0$, and fix other $\mu_j$ with $j\neq i$. Letting $\mu_i\to\yy$, we obtain the limiting profile of solution pair $(c_0,\Phi_0)$ of \eqref{1.3}-\eqref{1.4}.

(iii) We obtain the dynamics of \eqref{1.2} with $U(0,x)=(u_{i0}(x))$, namely, if {\bf(J2)} holds, then $C_*$ is asymptotic spreading speed of \eqref{1.2}; if {\bf(J2)} does not hold, then accelerated spreading happens for \eqref{1.2}. Moreover, if ${\bf(J^\gamma)}$ holds with $\gamma\in(1,2]$ and $m_0=m$, which implies that the accelerated spreading occurs, then more accurate longtime behaviors are obtained.

(iv) Choose a any $\mu_i>0$ and fix other $\mu_j$. We prove that the limiting problem of \eqref{1.1} is the problem \eqref{1.2} as $\mu_i\to\yy$.

Now we introduce our first main result.
\begin{theorem}\label{t1.1}Let $(u,g,h)$ be a solution of \eqref{1.1} and spreading happen. Then
\bess\left\{\begin{aligned}
&\lim_{t\to\yy}\max_{|x|\le ct}\sum_{i=1}^{m}|u_i(t,x)-u^*_i|=0 ~ {\rm for ~ any ~ } c\in(0,c_0) ~ ~ {\rm if ~ {\bf(J1)} ~ holds},\\
&\lim_{t\to\yy}\max_{|x|\le ct}\sum_{i=1}^{m}|u_i(t,x)-u^*_i|=0 ~ {\rm for ~ any ~ } c>0 ~ ~ {\rm if ~ {\bf(J1)} ~ does ~ not ~ hold},
\end{aligned}\right.\eess
where $c_0$ is uniquely determined by \eqref{1.3}-\eqref{1.4}.
\end{theorem}
\begin{remark}\label{r1.1}By Theorem \ref{A} and Theorem \ref{t1.1}, we know that if one of the kernel functions $J_i$ with $\mu_i>0$ violates
\[\dd\int_{0}^{\yy}xJ_i(x)\dx<\yy,\]
then the accelerated spreading will happen, which means that the species $u_i$ will accelerate the propagation of other species. This phenomenon is also captured by Xu et al \cite{XLL} for the cauchy problem, and is called by the transferability of acceleration propagation.
\end{remark}

Before giving our next main result, we further introduce an assumption on $F$, i.e.,\\
{\bf(f6)}\, For $i\in\{1,\cdots,m\}$, $\sum_{j=1}^{m}\partial_jf_i({\bf0})u^*_j>0$, $\sum_{j=1}^{m}\partial_jf_i({\bf u^*})u^*_j<0$ and $f_i(\eta {\bf u^*})>0$ with $\eta\in(0,1)$.

\begin{theorem}\label{t1.2}Assume that {\bf(f6)} holds, $m_0=m$ and ${\bf(J^\gamma)}$ holds with $\gamma\in(1,2]$. Let $(u,g,h)$ be a solution of \eqref{1.1} and spreading happen. Then
\bess\left\{\begin{aligned}
&\lim_{t\to\yy}\max_{|x|\le s(t)}\sum_{i=1}^{m}|u_i(t,x)-u^*_i|=0  ~ {\rm for ~ any } ~ 0\le s(t)=t^{\frac{1}{\gamma-1}}o(1) ~ ~ {\rm  ~ if ~}\gamma\in(1,2),\\
&\lim_{t\to\yy}\max_{|x|\le s(t)}\sum_{i=1}^{m}|u_i(t,x)-u^*_i|=0  ~ {\rm for ~ any } ~ 0\le s(t)=(t\ln t) o(1)  ~ ~ {\rm  ~ if ~ }\gamma=2.
\end{aligned}\right.\eess
\end{theorem}
\begin{remark}\label{r1.2}It can be seen from Theorem \ref{A}, Theorem \ref{t1.1}, Theorem \ref{t1.2}, \cite[Theorem 3.15]{WHD} and \cite[Theorem 1.2]{ZLN} that free boundary problem with nonlocal diffusion has richer dynamics than its counterpart with random diffusion. That is because the kernel function plays a important role in studying the dynamics of nonlocal diffusion problem, and the accelerated spreading may happen if kernel function violates the so-called "thin-tailed" condition , see \cite{XLL} and the references therein.
\end{remark}
Now we assume that {\bf (J1)} holds, and choose a any $\mu_i>0$ and fix other $\mu_j$. Denote the unique solution pair of \eqref{1.3}-\eqref{1.4} by $(c_{\mu_i},\Phi^{c_{\mu_i}})$. By the monotonicity of $\Phi^{c_{\mu_i}}$, there is a unique $l_{\mu_i}>0$ such that $\phi^{c_{\mu_i}}_i(-l_{\mu_i})=u^*_i/2$. Define $\hat{\Phi}^{c_{\mu_i}}(x)=\Phi^{c_{\mu_i}}(x-l_{\mu_i})$. Our next result concerns the limitation of $(c_{\mu_i},l_{\mu_i},\Phi^{c_{\mu_i}},\hat{\Phi}^{c_{\mu_i}})$ as $\mu_i\to\yy$.

\begin{theorem}\label{t1.3}If {\bf(J2)} holds, then $c_{\mu_i}\to C_*$, $l_{\mu_i}\to\yy$, $\Phi^{c_{\mu_i}}(x)\to{\bf0}$ and $\hat{\Phi}^{c_{\mu_i}}(x)\to\Psi(x)$ as $\mu_i\to\yy$,
where $(C_*,\Psi(x))$ is the minimal speed solution pair of travelling wave problem \eqref{1.5} with $\psi_i(0)=u^*_i/2$. If {\bf(J2)} does not hold, then $c_{\mu_i}\to \yy$ as $\mu_i\to\yy$.
\end{theorem}

We next study the dynamics for \eqref{1.2}. For every $1\le i\le m$, we denote the level set of solution component $U_i$ of \eqref{1.2} by $E^{i}_{\lambda}=\{x\in\mathbb{R}: U_i(t,x)=\lambda, ~ \lambda\in(0,u^*_i)\}$,
$x^{+}_{i,\lambda}=\sup E^i_{\lambda}$ and $x^{-}_{i,\lambda}=\inf E^i_{\lambda}$.
\begin{theorem}\label{t1.4} Let $U$ be a solution of \eqref{1.5} with $U(0,x)=(u_{i0}(x))$ and ${\bf \hat{u}}=\yy$. Then the following results hold:

{\rm(i)} If {\bf (J2)} holds, then
\bes\label{1.6}
\lim_{t\to\yy}\frac{|x^{\pm}_{i,\lambda}|}{t}=C_*, ~ ~ \lim_{|x|\to\yy}U(t,x)={\bf0} ~ ~ {\rm for ~ any ~ } t\ge0,
\ees
and
\bes\label{1.7}\left\{\begin{array}{lll}
\lim_{t\to\yy}\max_{|x|\le ct}\dd\sum_{i=1}^{m}|U_i(t,x)-u^*_i|=0~ ~ {\rm for ~ any ~ }c\in(0,C_*),\\
\lim_{t\to\yy}\max_{|x|\ge ct}\dd\sum_{i=1}^{m}|U_i(t,x)|=0~ ~ {\rm for ~ any ~ }c>C_*.
 \end{array}\right.
 \ees

{\rm(ii)} If {\bf (J2)} does not hold, then
\bes\label{1.8}\lim_{t\to\yy}\frac{|x^{\pm}_{i,\lambda}|}{t}=\yy.
\ees
Moveover,
\bes\label{1.9}
\lim_{t\to\yy}\max_{|x|\le ct}\sum_{i=1}^{m}|U_i(t,x)-u^*_i|=0~ ~ {\rm for ~ any ~ }c>0.
\ees

{\rm(iii)} If {\bf(f6)} holds, $m_0=m$ and ${\bf(J^\gamma)}$ holds with $\gamma\in(1,2]$. Then
\bess\left\{\begin{aligned}
&\lim_{t\to\yy}\max_{|x|\le s(t)}\sum_{i=1}^{m}|U_i(t,x)-u^*_i|=0  ~ {\rm for ~ any } ~ 0\le s(t)=t^{\frac{1}{\gamma-1}}o(1) ~ ~ {\rm  ~ if ~}\gamma\in(1,2),\\
&\lim_{t\to\yy}\max_{|x|\le s(t)}\sum_{i=1}^{m}|U_i(t,x)-u^*_i|=0  ~ {\rm for ~ any } ~ 0\le s(t)=(t\ln t) o(1)  ~ ~ {\rm  ~ if ~ }\gamma=2.
\end{aligned}\right.\eess
\end{theorem}

As before, we choose a any $\mu_i>0$ and fix other $\mu_j$. Our last main result concerns the limiting problem of \eqref{1.1} as $\mu_i\to\yy$.
\begin{theorem}\label{t1.5}

Problem \eqref{1.2} with $U(0,x)=(u_{i0}(x))$ for $|x|\le h_0$ and $U(0,x)={\bf 0}$ for $x\in\mathbb{R}\setminus [-h_0,h_0]$ is the limiting problem of \eqref{1.1} as $\mu_i\to\yy$. More precisely, denoting the unique solution of \eqref{1.1} by $(u_{\mu_i},g_{\mu_i},h_{\mu_i})$ and letting $\mu_i\to\yy$, we have $u_{\mu_i}(t,x)\to U(t,x)$ locally uniformly in $[0,\yy)\times\mathbb{R}$ and $-g_{\mu_i}(t), ~ h_{\mu_i}(t)\to\yy$ locally uniformly in $(0,\yy)$.
\end{theorem}

This paper is arranged as follows. In section 2, we prove Theorem \ref{t1.1} and Theorem \ref{t1.2}. Section 3 is devoted to the proofs of Theorem \ref{t1.3}, Theorem \ref{t1.4} and Theorem \ref{t1.5}. In Section 4, we will give two epidemic model as examples to
illustrate our results.

 \section{Proofs of Theorem \ref{t1.1} and Theorem \ref{t1.2}}
\setcounter{equation}{0} {\setlength\arraycolsep{2pt}
In this section, we will prove Theorem \ref{t1.1} and Theorem \ref{t1.2} by constructing some properly upper and lower solutions.

\begin{proof}[{\bf Proof of Theorem \ref{t1.1}:}]
Assume that {\bf(J1)} holds. For small $\ep>0$ and $K>0$, we define $\underline{h}(t)=c_0(1-2\ep)t+K$
and $\underline{U}(t,x)=(1-\ep)\left[\Phi(x-\underline{h}(t))+\Phi(-x-\underline{h}(t))-{\bf u^*}\right]$ for
$t\ge0$ and $x\in[-\underline{h}(t),\underline{h}(t)]$. It then follows from \cite[Lemma 3.4]{DN21} that for any small $\ep>0$,
there exist suitable $T>0$ and $K>0$ such that
\[
g(t+T)\le-\underline{h}(t), ~ ~ h(t+T)\ge\underline{h}(t), ~ ~ u(t+T,x)\succeq\underline{U}(t,x), ~ ~ ~ ~  t\ge0, ~ x\in[-\underline{h}(t),\underline{h}(t)].
\]
Direct calculations show
\bess
&&\max_{|x|\le c_0(1-3\ep)t}|\underline{U}(t,x)-(1-\ep){\bf u^*}|\\
&&\qquad=(1-\ep)\max_{|x|\le c_0(1-3\ep)t}\left\{\sum_{i=1}^{m}|\phi_i(x-\underline{h}(t))+\phi_i(-x-\underline{h}(t))-2u_i^*|^2\right\}^{\frac{1}{2}}\\
&&\qquad\le(1-\ep)\max_{|x|\le c_0(1-3\ep)t}\left\{\sum_{i=1}^{m}\bigg(|\phi_i(x-\underline{h}(t))-u_i^*|+|\phi_i(-x-\underline{h}(t))-u_i^*|\bigg)^2\right\}^{\frac{1}{2}}\\
&&\qquad=2(1-\ep)\left\{\sum_{i=1}^{m}|\phi_i(-c_0\ep t-K)-u_i^*|^2\right\}^{\frac{1}{2}}\to0 ~ ~ {\rm as} ~ t\to\yy.
\eess
Therefore, $\liminf_{t\to\yy}u(t,x)\succeq (1-\ep){\bf u^*}$ uniformly in $|x|\le c_0(1-3\ep)t$.
Then for any $c\in(0,c_0)$, by letting $\ep>0$ small sufficiently such that $c<c_0(1-3\ep)$,
we have $\liminf_{t\to\yy}u(t,x)\succeq (1-\ep){\bf u^*}$ uniformly in $|x|\le ct$.
In view of the arbitrariness of $\ep>0$, we obtain $\liminf_{t\to\yy}u(t,x)\succeq {\bf u^*}$ uniformly in $|x|\le ct$.
On the other hand, consider the following ODE system
\[\bar{u}_t=F(\bar{u}), ~ ~ ~ ~ \bar{u}(0)=(\|u_{i0}(x)\|_{C([-h_0,h_0])}).\]
By condition {\bf(f4)} and a comparison argument, we have $\limsup_{t\to\yy}u(t,x)\le{\bf u^*}$ uniformly in $\mathbb{R}$.
Therefore, the proof of the result with {\bf(J1)} holding is complete.

Next we consider the case {\bf(J1)} being violated. As in the proof of \cite[Theorem 1.3]{DN21}, for any integer $n\ge1$ and $i\in\{1,2,\cdots,m_0\}$, define
\begin{align*}
J^n_i(x)=\left\{\begin{aligned}
&J_i(x)& &{\rm if} ~ |x|\le n,\\
&\frac{n}{|x|}J_i(x)& &{\rm if} ~ |x|\ge n,
\end{aligned}\right.  ~ ~ \tilde{J}^n_i=\frac{J^n_i(x)}{\|J^n_i\|_{L^1(\mathbb{R})}}, ~ ~ {\bf J}^n=(J^n_i), ~ ~ {\rm and} ~ {\bf \tilde{J}}^n=(\tilde{J}^n_i)
  \end{align*}
 with $J^n_i(x)\equiv\tilde{J}^n_i\equiv0$ for $i\in\{m_0+1,\cdots,m\}$.
  Clearly, the following results about $J^n_i$ and $\tilde{J}^n_i$ hold:

  (1) $J^n_i(x)\le J_i(x)$, $|x|J^n_i(x)\le nJ_i(x)$, and for any $\alpha>0$,
  there is $c>0$ depending only on $n,\alpha,J_i$ such that $e^{\alpha|x|}J^n_i(x)\ge c e^{\frac{\alpha}{2}|x|}J_i(x)$ for $x\in\mathbb{R}$,
  which directly implies that
${\bf \tilde{J}}^n$ satisfies {\bf (J)} and {\bf(J1)}, but not {\bf(J2)}.

(2) ${\bf J}^n$ is nondecreasing in $n$, and $\lim_{n\to\yy}{\bf J}^n(x)={\bf \tilde J}^n(x)={\bf J}(x)$ in $[L^1(\mathbb{R})]^m$
 and locally uniformly in $\mathbb{R}$.\\
Let $(u^n,g^n,h^n)$ be the solution of the following problem
  \bess
\left\{\begin{aligned}
&u^n_{t}=D\circ\dd\int_{g^n(t)}^{h^{n}(t)}{\bf J}^n(x-y)\circ u^n(t,y)\dy-D\circ u^n+F(u^n), && t>0,~x\in(g^n(t),h^n(t)),\\
&u^n(t,x)=0, && t>0, x\notin(g^n(t),h^n(t)),\\
&{g^n}'(t)=-\sum_{i=1}^{m_0}\mu_i\dd\int_{g^n(t)}^{h^n(t)}\!\!\int_{h^n(t)}^{\infty}
J^n_i(x-y)u^n_i(t,x)\dy\dx, && t>0,\\
&{h^n}'(t)=\sum_{i=1}^{m_0}\mu_i\dd\int_{g^n(t)}^{h^n(t)}\!\!\int_{-\yy}^{g^n(t)}
J^n_i(x-y)u^n_i(t,x)\dy\dx, && t>0,\\
&u^n(0,x)=u(T,x),~g^n(T)=g(T), ~ h^n(0)=h(T), \ \ x\in[g(T),h(T)],
\end{aligned}\right.
 \eess
where $T>0$, $u^n=(u^n_i)$, and $(u,g,h)$ is the solution of \eqref{1.1}. For any integer $n\ge1$, it follows from \cite[Lemma 3.5]{DN21} that
\[g^n(t)\ge g(t+T), ~ h^n(t)\le h(t+T) ~ {\rm and } ~ u^n(t,x)\preceq u(t+T,x) ~ ~ {\rm for } ~ t\ge0,~ g^n(t)\le x\le h^n(t).\]
Since $F$ satisfies {\bf (f1)}-{\bf (f3)}, $F(w)+(D^n-D)w$ still satisfies {\bf (f1)}-{\bf (f3)}
with $D^n=(d_i\|J^n_i\|_{L^1(\mathbb{R})})$ and $n\gg1$. Denote the unique positive root of $F(w)+(D^n-D)\circ w=0$ by ${\bf u^*_n}$.
It is easy to see that $\lim_{n\to\yy}{\bf u^*_n}={\bf u^*}$. By \cite[Lemmas 3.6 and 3.8]{DN21}, the following problem
\bess\left\{\begin{array}{lll}
 D\circ\dd\int_{-\yy}^{0}{\bf J}^n(x-y)\circ \Phi(y)\dy-D\circ \Phi+c\Phi'(x)+F(\Phi)=0, ~ ~ -\yy<x<0,\\
 \Phi(-\yy)={\bf u^*_n}, ~ ~ \Phi(0)={\bf0}, ~ ~ c=\sum_{i=1}^{m_0}\mu_i\dd\int_{-\yy}^{0}\int_{0}^{\yy}J^n_i(x-y)\phi_i(x)\dy\dx
 \end{array}\right.
 \eess
 has a solution pair $(c^n,\Phi^n)$ with $\Phi^n(x)=(\phi^n_i(x))$ and $\lim_{n\to\yy}c^n=\yy$.

 As before, for small $\ep>0$ and $K>0$ we define
 \[\underline{h}^n(t)=c^n(1-2\ep)t+K,~ ~ \underline{u}^n(t,x)=(1-\ep)\left[\Phi^n(x-\underline{h}^n(t))+\Phi^n(-x-\underline{h}^n(t))-{\bf u^*_n}\right],\]
 with $t\ge0$ and $x\in[-\underline{h}^n(t),\underline{h}^n(t)]$. By \cite[Lemma 3.7]{DN21}, for small $\ep>0$ and all large $n\gg1$ there exist $K>0$ and $T>0$
 such that
\[
g(t+T)\le-\underline{h}^n(t), ~ ~ h(t+T)\ge\underline{h}^n(t), ~ ~ u(t+T,x)\succeq\underline{u}^n(t,x), ~ ~ ~ ~  t\ge0, ~ x\in[-\underline{h}^n(t),\underline{h}^n(t)].
\]
Similarly, we have that $\liminf_{t\to\yy}\underline  u(t,x)\succeq\liminf_{t\to\yy}\underline  u^n(t,x)\succeq (1-\ep){\bf u^*_n}$ uniformly in $|x|\le c^n(1-3\ep)t$.
Since $\lim_{n\to\yy}c^n=\yy$, for any $c>0$ there are large $N\gg1$ and small $\ep_0>0$ such that $c<c^n(1-3\ep)$ for $n\ge N$ and $\ep\in(0,\ep_0)$.
Thus $\liminf_{t\to\yy}\underline  u(t,x)\succeq (1-\ep){\bf u^*_n}$ uniformly in $|x|\le ct$.
Letting $n\to\yy$ and $\ep\to0$, we derive $\liminf_{t\to\yy}\underline  u(t,x)\succeq {\bf u^*}$ uniformly in $|x|\le ct$.
Together with our early conclusion, the desired result is obtained. The proof is ended.
\end{proof}

To prove Theorem \ref{t1.2}, the following two technical lemmas are crucial, and their proofs can be found from \cite{DN21} and \cite{DNwn}, respectively.
\begin{lemma}\label{l2.1} Let $P(x)$ satisfy {\bf (J)} and $\varphi_l(x)=l-|x|$ with $l>0$. Then for any $\ep>0$, there exists $L_{\ep}>0$ such that for any $l>L_{\ep}$,
\[\dd\int_{-l}^{l}P(x-y)\varphi_l(y)\dy\ge(1-\ep)\varphi_l(x) ~ ~ {\rm in ~ }[-l,l].\]
\end{lemma}
\begin{lemma}\label{l2.2} Let $P(x)$ satisfy {\bf (J)} and $l_2>l_1>0$. Define
\[\varphi(x)=\min\{1,\frac{l_2-|x|}{l_1}\}.\]
Then for any $\ep>0$, there is $L_{\ep}>0$ such that for all $l_1>L_{\ep}$ and $l_2-l_1>L_{\ep}$,
\[\dd\int_{-l_2}^{l_2}P(x-y)\varphi(y)\dy\ge(1-\ep)\varphi(x) ~ ~{\rm in ~ }[-l_2,l_2].\]
\end{lemma}

\begin{proof}[{\bf Proof of Theorem \ref{t1.2}:}]
Firstly, a simple comparison argument yields that $\limsup_{t\to\yy}u(t,x)\preceq {\bf u^*}$ uniformly in $\mathbb{R}$. Thus it remains to show the lower limitations of $u$. The discussion now is divided into two steps.

{\bf Step 1:} In this step, we prove $\liminf_{t\to\yy}u(t,x)\succeq {\bf u^*}$ uniformly in $[-s(t),s(t)]$ for any $0\le s(t)=t^{\frac{1}{\gamma-1}}o(1)$.

For small $\ep>0$, we define
\[\underline{h}(t)=(Kt+\theta)^{\frac{1}{\gamma-1}} ~ ~ {\rm and } ~ ~ \underline{u}(t,x)={\bf K_{\ep}}(1-\frac{|x|}{\underline{h}(t)}) ~ ~ ~ ~ {\rm for} ~ t\ge0, ~ x\in[-\underline{h}(t),\underline{h}(t)],\]
where ${\bf K_{\ep}}={\bf u^*}(1-\ep)$ and $K,\theta>0$ to be determined later.

Then we are going to verify that for any small $\ep>0$, there exist proper $K,T$ and $\theta$ such that
 \bes\label{2.1}
\left\{\begin{aligned}
&\underline u_t\preceq D\circ\dd\int_{-\underline h(t)}^{\underline h(t)}{\bf J}(x-y)\circ\underline u(t,y)\dy-D\circ\underline u+ F(\underline u), && t>0,~x\in(-\underline h(t),\underline h(t)),\\
&\underline u(t,\pm\underline h(t))\preceq {\bf0},&& t>0,\\
&\underline h'(t)\le\sum_{i=1}^{m_0}\mu_i\dd\int_{-\underline h(t)}^{\underline h(t)}\int_{\underline h(t)}^{\infty}
J_i(x-y)\underline u_i(t,x)\dy\dx,&& t>0,\\
&-\underline h'(t)\ge-\sum_{i=1}^{m_0}\mu_i\dd\int_{-\underline h(t)}^{\underline h(t)}\int_{-\yy}^{-\underline h(t)}
J_i(x-y)\underline u_i(t,x)\dy\dx,&& t>0,\\
&\underline h(0)\le h(T),\;\;\underline u(0,x)\preceq u(T,x),&& x\in[-\underline h(0),\underline h(0)].
\end{aligned}\right.
\ees
Once it is done, by comparison method we have
\[g(t+T)\le -\underline{h}(t), ~ ~ h(t+T)\ge\underline{h}(t) ~ ~ {\rm and } ~ ~ u(t+T,x)\succeq \underline{u}(t,x) ~ ~ {\rm for} ~ t\ge0, ~ x\in[-\underline{h}(t),\underline{h}(t)].\]
Moreover, for any $s(t)=t^{\frac{1}{\gamma-1}}o(1)$, direct computations show
\bess
\lim_{t\to\yy}\max_{|x|\le s(t)}\sum_{i=1}^{m}|\underline{u}_i(t,x)-(1-\ep)u^*_i|
=\lim_{t\to\yy}(1-\ep)\sum_{i=1}^{m}u^*_i\frac{s(t)}{\underline{h}(t)}=0,
\eess
which, together with our early conclusion and the arbitrariness of $\ep$, arrives at
\[\liminf_{t\to\yy}u(t,x)\succeq {\bf u^*} ~ ~ {\rm uniformly ~ in ~ }[-s(t),s(t)].\]
Hence the case $\gamma\in(1,2)$ is proved. The second inequality of \eqref{2.1} is obvious. We next show the third one. Simple calculations yield
\bess
&&\sum_{i=1}^{m_0}\mu_i\dd\int_{-\underline h(t)}^{\underline h(t)}\int_{\underline h(t)}^{\infty}
J_i(x-y)\underline u_i(t,x)\dy\dx\\
&&=\sum_{i=1}^{m_0}(1-\ep)\mu_iu^*_i\dd\int_{-\underline h(t)}^{\underline h(t)}\int_{\underline h(t)}^{\infty}
J_i(x-y)(1-\frac{|x|}{\underline{h}(t)})\dy\dx\\
&&\ge\sum_{i=1}^{m_0}(1-\ep)\mu_iu^*_i\dd\int_{0}^{\underline h(t)}\int_{\underline h(t)}^{\infty}
J_i(x-y)(1-\frac{x}{\underline{h}(t)})\dy\dx\\
&&=\sum_{i=1}^{m_0}\frac{(1-\ep)\mu_iu^*_i}{\underline{h}(t)}\dd\int_{-\underline h(t)}^{0}\int_{0}^{\infty}
J_i(x-y)(-x)\dy\dx\\
&&=\sum_{i=1}^{m_0}\frac{(1-\ep)\mu_iu^*_i}{\underline{h}(t)}\dd\int_{0}^{\underline h(t)}\int_{x}^{\infty}
J_i(y)x\dy\dx\\
&&=\sum_{i=1}^{m_0}\frac{(1-\ep)\mu_iu^*_i}{\underline{h}(t)}\bigg(\dd\int_{0}^{\underline h(t)}\int_{0}^{y}+\int_{\underline{h}(t)}^{\yy}\int_{0}^{\underline{h}(t)}\bigg)J_i(y)x\dx\dy\\
&&\ge\sum_{i=1}^{m_0}\frac{(1-\ep)\mu_iu^*_i}{2\underline{h}(t)}\int_{0}^{\underline h(t)}J_i(y)y^2\dy\ge\sum_{i=1}^{m_0}\frac{C_1(1-\ep)\mu_iu^*_i}{2\underline{h}(t)}\int_{\underline h(t)/2}^{\underline h(t)}y^{2-\gamma}\dy\\
&&\ge\tilde{C}_1(Kt+\theta)^{(2-\gamma)/(\gamma-1)}\ge\frac{K(Kt+\theta)^{(2-\gamma)/(\gamma-1)}}{\gamma-1}=\underline{h}'(t)
~ ~ {\rm if }~ K\le \tilde{C}_1(\gamma-1).
\eess
Since $J_i$ and $u$ are both symmetric in $x$, the fourth inequality of \eqref{2.1} also holds. Next we focus on the first one in \eqref{2.1}. We first claim that there is a constant $\hat{C}>0$ such that
\[\int_{-\underline h(t)}^{\underline h(t)}{\bf J}(x-y)\circ\underline u(t,y)\dy\succeq\hat{C}{\bf K_{\ep}}\underline h^{1-\gamma}(t) ~ ~ {\rm for } ~ t>0, ~ x\in[-\underline h(t),\underline h(t)].\]
Firstly, for $x\in[\underline{h}(t)/4,\underline{h}(t)]$ we have
\bess
&&\int_{-\underline h(t)}^{\underline h(t)}{\bf J}(x-y)\circ\underline u(t,y)\dy=\int_{-\underline h(t)-x}^{\underline h(t)-x}{\bf J}(y)\circ\underline u(t,x+y)\dy\\
&&\succeq {\bf K_{\ep}}\circ\int_{-\underline h(t)/4}^{-\underline h(t)/8}{\bf J}(y)(1-\frac{x+y}{\underline{h}(t)})\dy\succeq{\bf K_{\ep}}\circ\int_{-\underline h(t)/4}^{-\underline h(t)/8}{\bf J}(y)\frac{-y}{\underline{h}(t)}\dy\\
&&\succeq\frac{{\bf K_{\ep}}C_1}{\underline{h}(t)}\int_{\underline h(t)/8}^{\underline h(t)/4}y^{1-\gamma}\dy={\bf K_{\ep}}\hat C_1\underline{h}^{1-\gamma}(t).
\eess
For $x\in[0,\underline{h}(t)/4]$, we have
\bess
&&\int_{-\underline h(t)}^{\underline h(t)}{\bf J}(x-y)\circ\underline u(t,y)\dy=\int_{-\underline h(t)-x}^{\underline h(t)-x}{\bf J}(y)\circ\underline u(t,x+y)\dy\\
&&\succeq {\bf K_{\ep}}\circ\int_{\underline h(t)/8}^{\underline h(t)/4}{\bf J}(y)(1-\frac{x+y}{\underline{h}(t)})\dy\succeq \frac{{\bf K_{\ep}}}{\underline{h}(t)}\circ\int_{\underline h(t)/8}^{\underline h(t)/4}{\bf J}(y)y\dy\\
&&\succeq\frac{{\bf K_{\ep}}C_1}{\underline{h}(t)}\int_{\underline h(t)/8}^{\underline h(t)/4}y^{1-\gamma}\dy={\bf K_{\ep}}\hat C_1\underline{h}^{1-\gamma}(t).
\eess
Then our claim is verified since $J_i$ and $u$ are both even in $x$.

Define $\tilde{F}(\eta)=F(u^*_1\eta,\cdots,u^*_m\eta)$ for $\eta\in[0,1]$. By the assumptions on $F$, we easily show that there is a ${\bf C}\succ0$ such that
\[\tilde F(\eta)\succeq {\bf C}\min\{\eta,1-\eta\} ~ ~ {\rm for ~ any ~ }  \eta\in[0,1].\]
Hence there is a positive constant $\overline{C}$, depending only on $F$, such that
\bess
&&F(u^*_1(1-\ep)(1-\frac{|x|}{\underline{h}(t)}),\cdots,u^*_m(1-\ep)(1-\frac{|x|}{\underline{h}(t)}))\\
&&\succeq
(1-\frac{|x|}{\underline{h}(t)})F(u^*_1(1-\ep),\cdots,u^*_m(1-\ep))\\
&&\succeq (1-\frac{|x|}{\underline{h}(t)})\ep{\bf C}\succeq \overline{C}\ep\underline{u}.
\eess
By Lemma \ref{l2.1}, one can let $\theta$ large sufficiently such that
\[\int_{-\underline h(t)}^{\underline h(t)}{\bf J}(x-y)\circ\underline u(t,y)\dy\succeq (1-\ep^2)\underline{u}(t,x) ~ ~ {\rm for }~ t>0, ~ x\in[-\underline{h}(t),-\underline{h}(t)].\]
Therefore
\bess
&&D\circ\dd\int_{-\underline h(t)}^{\underline h(t)}{\bf J}(x-y)\circ\underline u(t,y)\dy-D\circ\underline u+ F(\underline u)\\
&&=\frac{\overline{C}\ep{\bf1}}{2}\circ\dd\int_{-\underline h(t)}^{\underline h(t)}{\bf J}(x-y)\circ\underline u(t,y)\dy+\bigg(D-\frac{\overline{C}\ep{\bf1}}{2}\bigg)\circ\int_{-\underline h(t)}^{\underline h(t)}{\bf J}(x-y)\circ\underline u(t,y)\dy-D\circ\underline u+ F(\underline u)\\
&&\succeq \frac{\overline{C}\ep{\bf1}}{2}\circ\dd\int_{-\underline h(t)}^{\underline h(t)}{\bf J}(x-y)\circ\underline u(t,y)\dy+(1-\ep^2)\bigg(D-\frac{\overline{C}\ep{\bf1}}{2}\bigg)\circ\underline{u}(t,x)-D\circ\underline u+ \overline{C}\ep\underline{u}\\
&&\succeq \frac{\overline{C}\ep{\bf1}}{2}\circ\dd\int_{-\underline h(t)}^{\underline h(t)}{\bf J}(x-y)\circ\underline u(t,y)\dy\succeq \frac{\overline{C}\ep}{2}{\bf K_{\ep}}\hat C_1\underline{h}^{1-\gamma}(t)\succeq \frac{K\underline{h}^{1-\gamma}(t)}{\gamma-1}{\bf K_{\ep}}\succeq \underline{u}_t
\eess
provided that $\ep$ and $K$ are small suitably. So the first inequality in \eqref{2.1} holds. For $K,\theta$ and $\ep$ as chosen above, noting that spreading happens we can find some $T>0$ such that
\[-\underline h(0)\ge g(T), ~ ~ \underline{h}(0)\le h(T), ~ ~ {\rm and } ~ ~ \underline{u}(0,x)\preceq{\bf K_{\ep}}\preceq u(T,x) ~ {\rm for ~ }x\in[-\underline{h}(0),\underline{h}(0)].\]
Therefore, \eqref{2.1} holds. So we finish Step 1.

{\bf Step 2:} We next show that $\liminf_{t\to\yy}u(t,x)\succeq {\bf u^*}$ uniformly in $[-s(t),s(t)]$ for any $0\le s(t)=t\ln t o(1)$.

For small $\ep>0$, we define
\[\underline{h}(t)=K(t+\theta)\ln(t+\theta) ~ ~ {\rm and } ~ ~ \underline{u}(t,x)={\bf K_{\ep}}\min\{1,\frac{\underline{h}(t)-|x|}{(t+\theta)^{1/2}}\} ~ ~ ~ ~ {\rm for} ~ t\ge0, ~ x\in[-\underline{h}(t),\underline{h}(t)],\]
where ${\bf K_{\ep}}={\bf u^*}(1-\ep)$ and $K,\theta>0$ to be determined later.

Now we are ready to prove
\bes\label{2.2}
\left\{\begin{aligned}
&\underline u_t\preceq D\circ\dd\int_{-\underline h(t)}^{\underline h(t)}{\bf J}(x-y)\circ\underline u(t,y)\dy-D\circ\underline u+ F(\underline u), &&\hspace{-7mm} t>0,x\in(-\underline h(t),\underline h(t))\setminus\{\underline{h}(t)-(t+\theta)^{1/2}\},\\
&\underline u(t,\pm\underline h(t))\preceq{\bf0},&&\hspace{-7mm} t>0,\\
&\underline h'(t)\le\sum_{i=1}^{m_0}\mu_i\dd\int_{-\underline h(t)}^{\underline h(t)}\int_{\underline h(t)}^{\infty}
J_i(x-y)\underline u_i(t,x)\dy\dx,&&\hspace{-7mm} t>0,\\
&-\underline h'(t)\ge-\sum_{i=1}^{m_0}\mu_i\dd\int_{-\underline h(t)}^{\underline h(t)}\int_{-\yy}^{-\underline h(t)}
J_i(x-y)\underline u_i(t,x)\dy\dx,&&\hspace{-4mm}~ t>0,\\
&\underline h(0)\le h(T),\;\;\underline u(0,x)\preceq u(T,x),&& \hspace{-7mm}x\in[-\underline h(0),\underline h(0)].
\end{aligned}\right.
\ees
If \eqref{2.2} holds, then by comparison argument we see
\[g(t+T)\le -\underline{h}(t), ~ ~ h(t+T)\ge\underline{h}(t) ~ ~ {\rm and } ~ ~ u(t+T,x)\succeq \underline{u}(t,x) ~ ~ {\rm for} ~ t\ge0, ~ x\in[-\underline{h}(t),\underline{h}(t)].\]
We first deal with the third inequality in \eqref{2.2}. Careful computations show
\bess
&&\sum_{i=1}^{m_0}\mu_i\dd\int_{-\underline h(t)}^{\underline h(t)}\int_{\underline h(t)}^{\infty}
J_i(x-y)\underline u_i(t,x)\dy\dx\\
&&\ge\sum_{i=1}^{m_0}\mu_iu^*_i(1-\ep)\dd\int_{0}^{\underline h(t)-(t+\theta)^{1/2}}\int_{\underline h(t)}^{\infty}
J_i(x-y)\dy\dx\\
&&=\sum_{i=1}^{m_0}\mu_iu^*_i(1-\ep)\dd\int_{-\underline h(t)}^{-(t+\theta)^{1/2}}\int_{0}^{\infty}
J_i(x-y)\dy\dx\\
&&=\sum_{i=1}^{m_0}\mu_iu^*_i(1-\ep)\dd\int_{(t+\theta)^{1/2}}^{\underline h(t)}\int_{x}^{\infty}
J_i(y)\dy\dx\\
&&=\sum_{i=1}^{m_0}\mu_iu^*_i(1-\ep)\dd\bigg(\int_{(t+\theta)^{1/2}}^{\underline h(t)}\int_{(t+\theta)^{1/2}}^{y}+\int_{\underline{h}(t)}^{\yy}\int_{(t+\theta)^{1/2}}^{\underline{h}(t)}\bigg)
J_i(y)\dx\dy\\
&&\ge\sum_{i=1}^{m_0}\mu_iu^*_i(1-\ep)\int_{(t+\theta)^{1/2}}^{\underline h(t)}\int_{(t+\theta)^{1/2}}^{y}J_i(y)\dx\dy\\
&&\ge\sum_{i=1}^{m_0}\mu_iu^*_i(1-\ep)C_1\int_{(t+\theta)^{1/2}}^{\underline h(t)}\frac{y-(t+\theta)^{1/2}}{y^2}\dy\\
&&\ge \sum_{i=1}^{m_0}C_1\mu_iu^*_i(1-\ep)\bigg(\ln \underline{h}(t)-\frac{\ln(t+\theta)}{2}+\frac{(t+\theta)^{1/2}}{\underline{h}(t)}-1\bigg)\\
&&\ge \sum_{i=1}^{m_0}C_1\mu_iu^*_i(1-\ep)\bigg(\frac{\ln(t+\theta)}{2}+\frac{1}{2}+\ln K+\ln(\ln (t+\theta))-\frac{3}{2}\bigg)\\
&&\ge \sum_{i=1}^{m_0}\frac{C_1\mu_iu^*_i(1-\ep)}{2}\big(\ln(t+\theta)+1\big)\ge K\big(\ln(t+\theta)+1\big)=\underline{h}'(t)
\eess
provided that $\theta$ is large enough and $K$ is small suitably. By the symmetry of $J_i$ and $\underline{u}$ on $x$, it is easy to show that the fourth one in \eqref{2.2} also holds.

Now we verify the first inequality of \eqref{2.2}, and first claim that for $x\in[-\underline{h}(t),-\underline{h}(t)+(t+\theta)^{1/2}]\cup[\underline{h}(t)-(t+\theta)^{1/2},\underline{h}(t)]$,
 there is a positive constant $\tilde{C}_1$ such that
 \[\int_{-\underline h(t)}^{\underline h(t)}{\bf J}(x-y)\circ\underline u(t,y)\dy\succeq\frac{\tilde{C}_1\ln(t+\theta)}{4(t+\theta)^{1/2}}{\bf K_{\ep}}.\]
 Direct calculations show
 \bess
 &&\int_{-\underline h(t)}^{\underline h(t)}{\bf J}(x-y)\circ\underline u(t,y)\dy\succeq{\bf K_{\ep}}\circ\int_{\underline h(t)-(t+\theta)^{1/2}}^{\underline h(t)}{\bf J}(x-y)\frac{\underline{h}(t)-y}{(t+\theta)^{1/2}}\dy\\
 &&=\frac{{\bf K_{\ep}}}{{(t+\theta)^{1/2}}}\circ\int_{\underline h(t)-(t+\theta)^{1/2}-x}^{\underline h(t)-x}{\bf J}(y)[\underline{h}(t)-x-y]\dy.
 \eess
 Then for $x\in[\underline{h}(t)-\frac{3}{4}(t+\theta)^{1/2},\underline{h}(t)]$, we have
 \bess
 &&\frac{{\bf K_{\ep}}}{{(t+\theta)^{1/2}}}\circ\int_{\underline h(t)-(t+\theta)^{1/2}-x}^{\underline h(t)-x}{\bf J}(y)[\underline{h}(t)-x-y]\dy\succeq \frac{{\bf K_{\ep}}}{{(t+\theta)^{1/2}}}\circ\int_{-(t+\theta)^{1/2}/4}^{-(t+\theta)^{1/4}/4}{\bf J}(y)(-y)\dy\\
 &&\succeq\frac{C_1{\bf K_{\ep}}}{{(t+\theta)^{1/2}}}\int_{-(t+\theta)^{1/2}/4}^{-(t+\theta)^{1/4}/4}(-y)^{-1}\dy=
 \frac{C_1{\bf K_{\ep}}}{(t+\theta)^{1/2}}\int_{(t+\theta)^{1/4}/4}^{(t+\theta)^{1/2}/4}y^{-1}\dy\\
 &&=\frac{C_1{\bf K_{\ep}}\ln(t+\theta)}{4(t+\theta)^{1/2}}
 \eess
 For $x\in[\underline{h}(t)-(t+\theta)^{1/2},\underline{h}(t)-\frac{3}{4}(t+\theta)^{1/2}]$, we obtain
 \bess
 &&\frac{{\bf K_{\ep}}}{{(t+\theta)^{1/2}}}\circ\int_{\underline h(t)-(t+\theta)^{1/2}-x}^{\underline h(t)-x}{\bf J}(y)[\underline{h}(t)-x-y]\dy\succeq \frac{{\bf K_{\ep}}}{(t+\theta)^{1/2}}\circ\int_{(t+\theta)^{1/4}/4}^{(t+\theta)^{1/2}/4}{\bf J}(y)[\underline{h}(t)-x-y]\dy\\
 &&\succeq\frac{{\bf K_{\ep}}C_1}{(t+\theta)^{1/2}}\int_{(t+\theta)^{1/4}/4}^{(t+\theta)^{1/2}/4}y^{-1}\dy=\frac{C_1{\bf K_{\ep}}\ln(t+\theta)}{4(t+\theta)^{1/2}}.
 \eess
 By the symmetry of $J_i$ and $\underline{u}$ again, we immediately complete the proof of our claim.

 By Lemma \ref{l2.2} with $l_2=\underline{h}(t)$ and $l_1=(t+\theta)^{1/2}$, we have
 \[\int_{-\underline h(t)}^{\underline h(t)}{\bf J}(x-y)\circ\underline u(t,y)\dy\succeq(1-\ep^2)\underline{u}(t,x) ~ ~ {\rm for ~ }t>0,~ x\in[-\underline{h}(t),\underline{h}(t)].\]
 Similarly to Step 1, there exists a positive constant $\overline{C}$ such that
 \[F(\underline{u})\succeq \overline{C}\ep \underline{u}~ ~ {\rm for ~ }t>0,~ x\in[-\underline{h}(t),\underline{h}(t)].\]
 Therefore, for $x\in[-\underline{h}(t),-\underline{h}(t)+(t+\theta)^{1/2}]\cup[\underline{h}(t)-(t+\theta)^{1/2},\underline{h}(t)]$, we similarly derive
 \bess
 &&D\circ\dd\int_{-\underline h(t)}^{\underline h(t)}{\bf J}(x-y)\circ\underline u(t,y)\dy-D\circ\underline u+ F(\underline u)\\
 &&\succeq \frac{\overline{C}\ep{\bf1}}{2}\circ\dd\int_{-\underline h(t)}^{\underline h(t)}{\bf J}(x-y)\circ\underline u(t,y)\dy\succeq\frac{\tilde{C}_1\overline{C}\ep\ln(t+\theta)}{8(t+\theta)^{1/2}}{\bf K_{\ep}}\\
 &&\succeq\frac{2K\ln(t+\theta)}{(t+\theta)^{1/2}}{\bf K_{\ep}}\succeq \underline{u}_t
 \eess
 provided that $\ep$ and $K$ are small enough, and $\theta$ is large suitably.
 For $|x|\le \underline{h}(t)-(t+\theta)^{1/2}$, we have
 \bess
D\circ\dd\int_{-\underline h(t)}^{\underline h(t)}{\bf J}(x-y)\circ\underline u(t,y)\dy-D\circ\underline u+ F(\underline u)\succeq \frac{\overline{C}\ep{\bf1}}{2}\circ\dd\int_{-\underline h(t)}^{\underline h(t)}{\bf J}(x-y)\circ\underline u(t,y)\dy\succeq{\bf0}\succeq\underline{u}_t.
 \eess
 Thus the first inequality of \eqref{2.2} is obtained.

Since spreading happens for \eqref{1.1}, for $\ep$, $\theta$ and $K$ as chosen above, we can choose $T>0$ properly such that $-\underline{h}(0)\ge g(T)$, $\underline{h}(0)\le h(T)$ and $\underline{u}(0,x)\preceq {\bf u^*}(1-\ep)\preceq u(T,x)$ for $x\in[-\underline{h}(0),\underline{h}(0)]$. So \eqref{2.2} is proved, and Step 2 is complete. The desired result directly follows from our early conclusions.
\end{proof}

\section{Proofs of Theorem \ref{t1.3}, Theorem \ref{t1.4} and Theorem \ref{t1.5}}
\setcounter{equation}{0} {\setlength\arraycolsep{2pt}
We first show the limitations of solution of semi-wave problem \eqref{1.3}-\eqref{1.4}, namely, Theorem \ref{t1.3}.

\begin{proof}[{\bf Proof of Theorem \ref{t1.3}:}]
We first prove the result with {\bf (J2)} holding. By some comparison considerations, we have that $c_{\mu_i}$ is nondecreasing in $\mu_i>0$.
Noticing $c_{\mu_i}<C_*$, we can define $C_{\yy}=\lim_{\mu_i\to\yy}c_{\mu_i}\le C_*$. Next we show that $\lim_{\mu_i\to\yy}l_{\mu_i}=\yy$. Clearly,
\bes\label{3.1}
0\le\int_{-\yy}^{0}\int_{0}^{\yy}J_i(x-y)\phi^{c_{\mu_i}}_i(x)\dy\dx\le \frac{c_{\mu_i}}{\mu_i}\le\frac{C_*}{\mu_i}.
\ees
Assume that $J_i$ does not have compact support set. Hence, for any $n>0$, by \eqref{1.4} we see
\bess
\frac{C_*}{\mu_i}&\ge&\int_{-\yy}^{0}\int_{0}^{\yy}J_i(x-y)\phi^{c_{\mu_i}}_i(x)\dy\dx\ge\int_{-n-1}^{-n}\phi^{c_{\mu_i}}_i(x)\int_{n+1}^{\yy}J_i(y)\dy\dx\\
&\ge&\phi^{c_{\mu_i}}_i(-n)\int_{n+1}^{\yy}J_i(y)\dy\ge0,
\eess
which implies that $\lim_{\mu_i\to\yy}\phi^{c_{\mu_i}}_i(-n)=0$. Noting that $\phi^{c_{\mu_i}}_i(x)$ is decreasing in $x\le0$,
we have that $\lim_{\mu_i\to\yy}\phi^{c_{\mu_i}}_i(x)=0$ locally uniformly in $(-\yy,0]$, which yields $\lim_{\mu_i\to\yy}l_{\mu_i}=\yy$.

Assume that $J_i$ is compactly supported, and $[-L,L]$ is the smallest set which contains the support set of $J_i$.
Combining \eqref{3.1} and the uniform boundedness of ${\phi^{c_{\mu_i}}_i}'(x)$,
one easily has that $\lim_{\mu_i\to\yy}\phi^{c_{\mu_i}}_i(x)=0$ locally uniformly in $[-L,0]$. Since ${\Phi^{c_{\mu_i}}}'$ is uniformly bounded about $\mu_i>1$,
it follows from a compact argument that there is a sequence $\mu^n_i\to\yy$ and a nonincreasing function $\Phi_{\yy}=(\phi^{\yy}_i)\in [C((-\yy,0])]^m$
such that $\Phi^{c_{\mu^n_i}}\to\Phi_{\yy}$ locally uniformly in $(-\yy,0]$ as $n\to\yy$. Clearly, $\Phi_{\yy}\in[{\bf0},{\bf u^*}]$. By the dominated convergence theorem,
\[D\circ\dd\int_{-\yy}^{0}{\bf J}(x-y)\circ \Phi_{\yy}(t,y)\dy-D\circ \Phi_{\yy}+c\Phi_{\yy}'(x)+F(\Phi_{\yy})=0, ~ ~ -\yy<x<0.\]
Thus
 \bes\label{3.2}
 d_i\dd\int_{-\yy}^{0}J_i(x-y)\phi^{\yy}_i(y)\dy-d_i\phi^{\yy}_i+c_{\mu_i}{\phi^{\yy}_i}'+f_i(\phi^{\yy}_1,\phi^{\yy}_2,\cdots,\phi^{\yy}_m)=0, \quad -\yy<x<0.
 \ees
Moreover, $\phi^{\yy}_i(x)=0$ in $[-L,0]$. If $\phi^{\yy}_i(x)\not\equiv0$ for $x\le0$, there exists $L_1\le-L$ such that $\phi^{\yy}_i(L_1)=0<\phi^{\yy}_i(x)$ in $(-\yy,L_1)$.
By \eqref{3.2}, {\bf(J)} and the assumptions on $F$, we have
\[0=d_i\int_{-\yy}^{0}J_i(L_1-y)\phi^{\yy}_i(y)\dy+f_i(\underbrace{\phi^{\yy}_1(L_1),\cdots,0}_{i},\cdots,\phi^{\yy}_m(L_1))>0,\]
which implies that $\phi^{\yy}_i(x)\equiv0$ for $x\le0$. Hence $\lim_{\mu_i\to\yy}l_{\mu_i}=\yy$.

Notice that $\hat{\Phi}^{c_{\mu_i}}$ and $\hat\Phi^{c_{\mu_i}}$$'$ are uniformly bounded for $\mu_i>1$ and $x\le -l_{\mu_i}$.
By a compact consideration again, for any sequence $\mu^n_i\to\yy$, there exists a subsequence, denoted by itself, such that
  $\lim_{n\to\yy}\hat{\Phi}^{c_{\mu^n_i}}=\hat{\Phi}^{\yy}$ locally uniformly in $\mathbb{R}$
  for some nonincreasing and continuous function $\hat{\Phi}^{\yy}\in[{\bf0},{\bf u^*}]$.
  Moreover, $\hat{\Phi}^{\yy}(0)=(u^*_1/2,\hat{\phi}^{\yy}_2(0),\cdots, \hat{\phi}^{\yy}_m(0))$. Using the dominated convergence theorem again, we have
  \[D\circ\dd\int_{\mathbb{R}}{\bf J}(x-y)\circ \hat\Phi^{\yy}(y)\dy-D\circ \hat\Phi^{\yy}+C_{\yy}\hat{\Phi}^{\yy '}+F(\hat\Phi^{\yy})=0, ~ ~ -\yy<x<\yy.\]
Together with the properties of $\hat{\Phi}^{\yy}$ and the assumptions on $F$,
we easily derive that $\hat{\Phi}^{\yy}(-\yy)={\bf u^*}$ and $\hat{\Phi}^{\yy}(\yy)={\bf 0}$. Thus, $(C_{\yy},\hat{\Phi}^{\yy})$ is a solution pair of \eqref{1.5}. By Theorem \ref{B},
$C_*$ is the minimal speed of \eqref{1.5}. Noticing that $C_{\yy}\le C_*$, we derive that $C_{\yy}=C_*$ and $\hat{\Phi}^{\yy}=\Psi$.
By the arbitrariness of sequence $\mu^n_{i}$, we have that $\hat{\Phi}^{c_{\mu_i}}(x)\to\Psi(x)$ locally uniformly in $\mathbb{R}$ as $\mu_i\to\yy$.

We now show that if {\bf(J2)} does not hold, then $c_{\mu_i}\to\yy$ as $\mu_i\to\yy$.
Since $c_{\mu_i}$ is nondecreasing in $\mu_i>0$, we have that $\lim_{\mu_i\to\yy}c_{\mu_i}:=C_{\yy}\in(0,\yy]$.
Arguing indirectly, we assume that $C_{\yy}\in(0,\yy)$. Then following the similar lines in previous arguments,
we can prove that \eqref{1.5} has a solution pair $(C_{\yy},\Phi_{\yy})$ with $\Phi_{\yy}$ nonincreasing,
$\Phi_{\yy}(-\yy)={\bf u^*}$ and $\Phi_{\yy}(\yy)={\bf 0}$. This is a contradiction to Theorem \ref{B}. So $C_{\yy}=\yy$. The proof is complete.
\end{proof}
Then we give the proof of Theorem \ref{t1.4}.
\begin{proof}[{\bf Proof of Theorem \ref{t1.4}:}]

 (i) Since {\bf (J2)} holds, problem \eqref{1.5} has a solution pair $(C_*,\Psi_{C_*})$ with $C_*>0$ and $\Psi_{C_*}$ nonincreasing in $\mathbb{R}$.
We first claim that $\Psi_{C_*}=(\psi_i)\succ{\bf 0}$ and $\Psi_{C_*}$ is monotonically decreasing in $\mathbb{R}$. For any $i\in\{1,2,\cdots,m_0\}$ and $l>0$,
define $\tilde{\psi_i}(x)=\psi_i(x-l)$. Then applying \cite[Lemma 2.2]{DN21} to $\tilde\psi$, we derive that $\psi_i(x)>0$ for $x<l$. By the arbitrariness
of $l>0$, we have $\psi_i(x)>0$ for $x\in\mathbb{R}$. Then for $i\in\{m_0+1,\cdots,m_0\}$, it follows from the assumptions on $F$ that $\psi'_i(x)<0$ in $\mathbb{R}$.
Thus $\psi_i(x)>0$ in $\mathbb{R}$ for $i\in\{m_0+1,\cdots,m_0\}$. To show the monotonicity of $\Psi_{C_*}$, it remains to verify that $\psi_i$ is decreasing in
$\mathbb{R}$ for every $i\in\{1,\cdots,m_0\}$. For any $\delta>0$, define $w(x)=\psi_i(x-\delta)-\psi_i(x)$ for any $i\in\{1,\cdots,m_0\}$.
Clearly, $w(x)\ge0$ in $\mathbb{R}$ and $w(x)\not\equiv0$ for $x<0$. By \eqref{1.5}, $w(x)$ satisfies
\[d_i\int_{-\yy}^{\yy}J_i(x-y)w(y)\dy-d_iw(x)+C_*w'(x)+q(x)w(x)\le0, ~ ~ x\in\mathbb{R}.\]
By \cite[Lemma 2.5]{DLZ}, $w(x)>0$ in $x<0$, and so $\psi_i(x)$ is decreasing in $x<0$. As before, for any $l>0$,
define $\tilde{\psi}_i(x)=\psi_i(x-l)$. Similarly, we can show that $\psi_i(x)$ is decreasing in $x<l$. Thus, our claim is verified.

Define $\bar{U}=K\Psi_{C_*}(x-C_*t)$ with $K\gg1$. We next show that $\bar{U}$ is an upper solution of \eqref{1.2}. In view of the assumptions on $U(0,x)$ and our above analysis,
 there is $K\gg1$ such that $\bar{U}(0,x)=K\Psi_{C_*}(x)\succeq U(0,x)$ in $\mathbb{R}$. Moreover, by {\bf (f1)},
 we have $KF(\Psi_{C_*}(x-C_*t))\succeq F(K\Psi_{C_*}(x-C_*t))$, and thus
 \bess
 \bar{U}_t=-C_*K\Psi'_{C_*}(x-C_*t)\succeq D\circ\int_{-\yy}^{\yy}{\bf J}(x-y)\circ \bar{U}(t,y)\dy-D\circ \bar{U}+F(\bar{U}).
 \eess
It then follows from a comparison argument that $U(t,x)\preceq\bar{U}(t,x)$ for $t\ge0$ and $x\in\mathbb{R}$.
Noticing the properties of  $\psi_i$, for any $\lambda\in(0,u^*_i)$ there is a unique $y_*\in\mathbb{R}$ such that $K\psi_i(y_*)=\lambda$. Therefore, it is easy to
 see that
 \bes \label{3.3}
 x^{-}_{i,\lambda}(t)\le x^{+}_{i,\lambda}(t)\le y_*+C_*t.
 \ees

 Similarly, we can prove that for suitable $K_1\gg1$, $K_1\Psi(-x-C_*t)$ is also an upper solution of \eqref{1.5},
 and there is a unique $\tilde y_*\in\mathbb{R}$ such that $K_1\psi_i(\tilde y_*)=\lambda$.
 Then one easily derive $-\tilde y_*-C_*t\le x^{-}_{i,\lambda}(t)\le x^{+}_{i,\lambda}(t)$. This together with \eqref{3.3} arrives at
 \[\limsup_{t\to\yy}\frac{|x^{-}_{i,\lambda}(t)|}{t}\le \limsup_{t\to\yy}\frac{|x^{+}_{i,\lambda}(t)|}{t}\le C_*.\]

 Next we claim that
  \[\liminf_{t\to\yy}\frac{|x^{+}_{i,\lambda}(t)|}{t}\ge \liminf_{t\to\yy}\frac{|x^{-}_{i,\lambda}(t)|}{t}\ge C_*.\]
We assume $\mu_1>0$, and fix other $\mu_i$. Denote the unique solution of \eqref{1.1} by $(u^{\mu_1},g^{\mu_1},h^{\mu_1})$. By a comparison consideration, we have
$U(t,x)\succeq u^{\mu_1}$ in $[0,\yy)\times[g^{\mu_1}(t),h^{\mu_1}(t)]$ for any $\mu_1>0$.
Moreover, we can choose $\mu_1$ large sufficiently, say $\mu_1>\tilde{\mu}>0$,
so that spreading happens for $(u^{\mu_1},g^{\mu_1},h^{\mu_1})$. Moreover, by Theorem \ref{A}, we have
\[\lim_{t\to\yy}\frac{-g^{\mu_1}}{t}=\lim_{t\to\yy}\frac{h^{\mu_1}}{t}=c_0.\]
To stress the dependence of $c_0$ on $\mu_1$, we rewrite $c_0$ as $c^{\mu_1}$. By Theorem \ref{t1.3}, $\lim_{\mu_1\to\yy}c^{\mu_1}=C_*$.
For any $\lambda\in(0,u^*_i)$, we choose $\delta$ small enough such that $\lambda<u^*_i-\delta$. Then by virtue of Theorem \ref{t1.1}, for any $0<\ep\ll1$, there is $T>0$ such that
\[\lambda<u^*_i-\delta\le u^{\mu_1}_i\le u^*_i+\delta, ~ ~ {\rm for} ~ t\ge T, ~ ~ |x|\le (c^{\mu_1}-\ep)t,\]
which obviously implies that $x^{-}_{i,\lambda}(t)\le-(c^{\mu_1}-\ep)t$ and $x^{+}_{i,\lambda}(t)\ge (c^{\mu_1}-\ep)t$.
Due to the arbitrariness of $\ep$ and $\mu_1$, our claim is proved. Since $K\Psi_{C_*}(x-C_*t)$ and $K_1\Psi(-x-C_*t)$ are the upper solution of \eqref{1.5},
it is easy to prove the second limitation of \eqref{1.6}. Thus \eqref{1.6} is obtained.

Now we prove \eqref{1.7}. Let $\bar{u}$ be the solution of
\[\bar{u}_t=F(\bar{u}), ~ ~ ~ ~ \bar{u}(0)=(\max\{\|u_{i0}(x)\|_{C([-h_0,h_0])},u^*_i\}).\]
By {\bf (f4)} and comparison principle, we derive that $\limsup_{t\to\yy}U(t,x)\preceq {\bf u^*}$ uniformly in $\mathbb{R}$. As before, for any $c\in (0,C_*)$, let $\mu_1>\tilde{\mu}$ large enough such that $c<c^{\mu_1}$. Using Theorem \ref{t1.1} and comparison principle, we see
$\liminf_{t\to\yy}U(t,x)\succeq {\bf u^*}$ uniformly in $x\in[-ct,ct]$ which, combined with our early conclusion, yields the desired result.

Moreover, since $K\Psi_{C_*}(x-C_*t)\succeq U(t,x)$ and $K_1\Psi(-x-C_*t)\succeq U(t,x)$ for $t\ge0$ and $x\in\mathbb{R}$, we have for any $c>C_*$,
\bess
0&\le&\sup_{|x|\ge ct}\dd\sum_{i=1}^{m}U_i(t,x)\le\sup_{x\ge ct}\dd\sum_{i=1}^{m}U_i(t,x)+\sup_{x\le -ct}\dd\sum_{i=1}^{m}U_i(t,x)\\
&\le&\sup_{x\ge ct}\dd\sum_{i=1}^{m}K\psi_i(x-C_*t)+\sup_{x\le -ct}\dd\sum_{i=1}^{m}K_1\psi_i(-x-C_*t)\\
&=&(K+K_1)\dd\sum_{i=1}^{m}\psi_i(ct-C_*t)\to0 ~  {\rm as } ~ t\to\yy.
\eess
Therefore, conclusion (i) is proved.

(ii) We now assume that {\bf (J2)} does not hold, but {\bf (J1)} holds. By Theorem \ref{t1.3}, $\lim_{\mu_1\to\yy}c^{\mu_1}=\yy$.
Thanks to the above arguments, we have $x^{-}_{i,\lambda}(t)\le-(c^{\mu_1}-\ep)t$ and $x^{+}_{i,\lambda}(t)\ge (c^{\mu_1}-\ep)t$.
Letting $\mu_1\to\yy$ and $\ep\to0$, we have $\lim_{t\to\yy}\frac{|x^{\pm}_{i,\lambda}|}{t}=\yy$. We next prove \eqref{1.9}. For any $c>0$, let $\mu_1$ large
enough such that $c^{\mu_1}>c$ and spreading happens for $(u^{\mu_1},g^{\mu_1},h^{\mu_1})$.
By a comparison argument and Theorem \ref{t1.1}, we see $\liminf_{t\to\yy}U(t,x)\succeq {\bf u^*}$ uniformly in $|x|\le ct$. Together with our previous result,
we obtain \eqref{1.9}.

We now suppose that {\bf (J1)} does not hold. It then follows from Theorem \ref{t1.1} that for any $c>0$, there is $T>0$ such that
\[\lambda<u^*_i-\delta\le u^{\mu_1}_i\le u^*_i+\delta, ~ ~ {\rm for} ~ t\ge T, ~ ~ |x|\le ct,\]
which clearly indicates \eqref{1.8}. As for \eqref{1.9}, by Theorem \ref{t1.1} again, we see $\liminf_{t\to\yy}U(t,x)\succeq {\bf u^*}$ uniformly in $|x|\le ct$.

(iii) As above, $(u^{\mu_1},g^{\mu_1},h^{\mu_1})$ is a lower solution to \eqref{1.2}. By Theorem \ref{t1.2} and our early conclusion, we immediately derive the desired result. Thus the proof is complete.
\end{proof}
Finally, we show the proof of Theorem \ref{t1.5}.
\begin{proof}[{\bf Proof of Theorem \ref{t1.5}:}] It clearly follows from comparison principles (see \cite[Lemmas 3.1 and 3.2]{DN21}) that $(u_{\mu_i},g_{\mu_i},h_{\mu_i})$ is monotonically nondecreasing in $\mu_i>0$. Hence we can define $G(t)=\lim_{\mu_i\to\yy}g_{\mu_i}(t)\in[-\yy,-h_0]$, $H(t)=\lim_{\mu_i\to\yy}h_{\mu_i}(t)\in[h_0,\yy]$ and $\hat{U}(t,x)=\lim_{\mu_i\to\yy}u_{\mu_i}(t,x)\le U(t,x)$ for $t>0$ and $G(t)<x<H(t)$. Here $u_{\mu_i}=(u^j_{\mu_i})$ and $\hat{U}=(\hat{U}_j)$ with $j\in\{1,2,\cdots,m\}$.

We now claim that $G(t)=-\yy$ and $H(t)=\yy$ for all $t>0$. We only prove the former since the latter can be handled by similar arguments. Arguing indirectly, assume that there is $t_0>0$ such that $G(t_0)>-\yy$. Then $-h_0\ge g_{\mu_i}(t)\ge G(t)\ge G(t_0)>-\yy$ for $t\in(0,t_0]$. By {\bf(J)}, there are small $\ep_1,\delta>0$ such that $J_i(|x|)>\ep_1$ for $|x|\le2\delta$. Therefore, for $t\in(0,t_0]$, we have
\bess
g'_{\mu_i}(t)&=&-\sum_{j=1}^{m_0}\mu_j\int_{g_{\mu_i}(t)}^{h_{\mu_i}(t)}\int_{-\yy}^{g_{\mu_i}(t)}J_j(x-y)u^j_{\mu_i}(t,x){\rm d}y{\rm d}x\\
&\le&-\mu_i\int_{g_{\mu_i}(t)}^{h_{\mu_i}(t)}\int_{-\yy}^{g_{\mu_i}(t)}J_i(x-y)u^i_{\mu_i}(t,x){\rm d}y{\rm d}x\le-\mu_i\int_{g_{\mu_i}(t)}^{g_{\mu_i}(t)+\delta}\int_{g_{\mu_i}(t)-\delta}^{g_{\mu_i}(t)}J_i(x-y)u^i_{\mu_i}(t,x){\rm d}y{\rm d}x\\
&\le&-\mu_i\ep_1\delta\int_{g_{\mu_i}(t)}^{g_{\mu_i}(t)+\delta}u^i_{\mu_i}(t,x){\rm d}x.
\eess
By the dominated convergence theorem, we see
\[\lim_{\mu_i\to\yy}\int_{g_{\mu_i}(t)}^{g_{\mu_i}(t)+\delta}u^i_{\mu_i}(t,x){\rm d}x=\int_{G(t)}^{G(t)+\delta}\hat{U}_i(t,x){\rm d}x>0 ~ ~ {\rm for } ~ t\in(0,t_0].\]
Then
\[-\frac{g_{\mu_i}(t_0)+h_0}{\mu_i}\ge\ep_1\delta\int_{0}^{t_0}\int_{g_{\mu_i}(t)}^{g_{\mu_i}(t)+\delta}u^i_{\mu_i}(t,x){\rm d}x\dt\to \ep_1\delta\int_{0}^{t_0}\int_{G(t)}^{G(t)+\delta}\hat{U}_i(t,x){\rm d}x\dt>0 ~ ~ {\rm as } ~ \mu_i\to\yy,\]
which clearly implies that $G(t_0)=-\yy$. So $G(t)=-\yy$ for all $t>0$, and our claim is verified. Combining with the monotonicity of $g_{\mu_i}(t)$ and $h_{\mu_i}(t)$ in $t$, one easily shows that $\lim_{\mu_i\to\yy}-g_{\mu_i}(t)=\lim_{\mu_i\to\yy}h_{\mu_i}(t)\to\yy$ locally uniformly in $(0,\yy)$.

Next we prove that $\hat{U}$ satisfies \eqref{1.2}. For any $(t,x)\in(0,\yy)\times\mathbb{R}$, there are large $\hat\mu_i>0$ and $t_1<t$ such that $x\in(g_{\mu_i}(s),h_{\mu_i}(s))$ for all $\mu_i\ge\hat\mu_i$ and $s\in[t_1,t]$. Integrating the first $m$ equations in \eqref{1.1} over $t_1$ to $s\in(t_1,t]$ yields that for $j\in\{1,2,\cdots,m_0\}$,
\[u^j_{\mu_i}(s,x)-u^j_{\mu_i}(t_1,x)=\int_{t_1}^{s}\bigg(d_j\mathcal{L}[u^j_{\mu_i}](\tau,x)+f_j(u^1_{\mu_i}(\tau,x),u^2_{\mu_i}(\tau,x),\cdots,u^m_{\mu_i}(\tau,x))\bigg){\rm d}\tau,\]
and
\[u^j_{\mu_i}(s,x)-u^j_{\mu_i}(t_1,x)=\int_{t_1}^{s}f_j(u^1_{\mu_i}(\tau,x),u^2_{\mu_i}(\tau,x),\cdots,u^m_{\mu_i}(\tau,x)){\rm d}\tau ~ ~ {\rm for } ~ j\in\{m_0+1,\cdots,m\}.\]
Letting $\mu_i\to\yy$ and using the dominated convergence theorem, we have that for $j\in\{1,2,\cdots,m_0\}$,
\[\hat{U}_j(s,x)-\hat{U}_j(t_1,x)=\int_{t_1}^{s}\bigg(d_j\mathcal{L}[\hat{U}_j](\tau,x)+f_j(\hat{U}_1(\tau,x),\hat{U}_2(\tau,x),\cdots,\hat{U}_m(\tau,x))\bigg){\rm d}\tau,\]
and
\[\hat{U}_j(s,x)-\hat{U}_j(t_1,x)=\int_{t_1}^{s}f_j(\hat{U}_1(\tau,x),\hat{U}_2(\tau,x),\cdots,\hat{U}_m(\tau,x)){\rm d}\tau ~ ~ {\rm for } ~ j\in\{m_0+1,\cdots,m\}.\]
Then differentiating the above equations by $s$, we see that $\hat{U}$ solves \eqref{1.2} for any $(t,x)\in(0,\yy)\times\mathbb{R}$. Moreover,
since $0\le\hat{U}(t,x)\le U(t,x)$ in $(0,\yy)\times\mathbb{R}$, it is easy to see that $\lim_{t\to0}\hat{U}(t,x)={\bf 0}$. By the uniqueness of solution to \eqref{1.2}, $\hat{U}(t,x)\equiv U(t,x)$ in $[0,\yy)\times\mathbb{R}$.
Using Dini's theorem, we directly derive our desired results.
\end{proof}

\section{Examples}
In this section, we introduce two epidemic models to explain our conclusions.

{\bf Example 1.} To investigate the spreading of some infectious diseases, such as cholera, Capasso and Maddalena \cite{CM} studied the following problem
\bes\label{4.1}
\left\{\begin{aligned}
&\partial_tu_1=d_1\Delta u_1-au_1+cu_2, & &t>0,~x\in\Omega,\\
&\partial_tu_2=d_2\Delta u_2-bu_2+G(u_1), & &t>0,~x\in\Omega,\\
&\frac{\partial u_1}{\partial \nu}+\lambda_1u_1=0, ~ ~ \frac{\partial u_2}{\partial \nu}+\lambda_2u_2=0, & & t>0,~ x\in\partial \Omega,
\end{aligned}\right.
 \ees
 where $u_1$ and $u_2$ represent the concentration of the infective agents, such as bacteria, and the infective human population, respectively. Both of them adopt the random diffusion (or called by local diffusion) strategy. And  $d_1$ and $d_2$ are their
respective diffusion rates. $-au_1$ is the death rate of the infection agents, $cu_2$ is the
growth rate of the agents contributed by the infective humans, and $-bu_2$ is the death rate of the infective human population. The function $G(u_1)$ describes the infective rate of humans, and its assumptions will be given later.

Recently, much research for model \eqref{4.1} and its variations has been done. For example, one can refer to \cite{ABL,LZW} for the free boundary problem with local diffusion, and \cite{ZLN} for the spreading speed. Particularly, Zhao et al \cite{ZZLD} recently replaced the local diffusion term of $u_1$ with the nonlocal diffusion operator as in \eqref{1.1}, and assume $d_2=0$. They found that the dynamics of their model is different that of \cite{ABL}.

Here we assume that the dispersal of the infective human population is also approximated by the nonlocal diffusion as in \eqref{1.1}, and thus propose the following model
\bes\left\{\begin{aligned}\label{4.2}
&\partial_tu_1=d_1\int_{g(t)}^{h(t)}J_1(x-y)u_1(t,y)\dy-d_1u_1-au_1+cu_2, & &t>0,~x\in(g(t),h(t)),\\
&\partial_tu_2=d_2\int_{g(t)}^{h(t)}J_2(x-y)u_2(t,y)\dy-d_2u_2-bu_2+G(u_1), & &t>0,~x\in(g(t),h(t)),\\
&u_i(t,g(t))=u_i(t,h(t))=0,& &t>0, ~ i=1,2,\\
&g'(t)=-\sum_{i=1}^{2}\mu_i\int_{g(t)}^{h(t)}\int_{-\yy}^{g(t)}
J_i(x-y)u_i(t,x)\dy\dx,& &t>0,\\
&h'(t)=\sum_{i=1}^{2}\mu_i\int_{g(t)}^{h(t)}\int_{h(t)}^{\infty}
J_i(x-y)u_i(t,x)\dy\dx,& &t>0,\\
&-g(0)=h(0)=h_0>0;\;\;u_1(0,x)=u_{10}(x),\;\;u_2(0,x)=u_{20}(x),&&x\in[-h_0,h_0],
 \end{aligned}\right.
 \ees
 where $J_i$ satisfy {\bf(J)}, $d_i,a,b,c$ are positive constant, $\mu_i\ge0$ and $\sum_{i=1}^{2}\mu_i>0$. $G(z)$ satisfies

(i) $G\in C^1([0,\yy))$, $G(0)=0$, $G'(z)>0$ for $z\ge0$ and $G'(0)>\frac{ab}{c}$;

(ii) $\bigg(\frac{G(z)}{z}\bigg)'<0$ for $z>0$ and $\lim_{t\to\yy}\frac{G(z)}{z}<\frac{ab}{c}$; Assumptions (i) and (ii) clearly imply that there is a unique positive constant $u^*_1$ such that $\frac{G(u^*_1)}{u^*_1}=\frac{ab}{c}$. Define $u^*_2=\frac{G(u^*_1)}{b}$.

(iii) $\bigg(\frac{G(u^*_1)}{u^*_1}\bigg)'<-\frac{ab}{c}$;\\
An example for $G$ is $G(z)=\frac{\alpha z}{1+\beta z}$ with $\alpha>\frac{ab}{c}$ and $\beta>\frac{\alpha c}{ab}$. Obviously, we have
\[F(u_1,u_2)=(f_1(u_1,u_2),f_2(u_1,u_2))=(-au_1+cu_2, -bu_2+G(u_1)) ~ ~{\rm and ~ ~ }{\bf u^*}=(u^*_1,u^*_2).\]
 By the similar methods in \cite{DNwn}, we easily show the dynamics of \eqref{4.2} are governed by the following spreading-vanishing dichotomy

\sk{\rm(i)}\, \underline{Spreading:} $\lim_{t\to\yy}-g(t)=\lim_{t\to\yy}h(t)=\yy$ (necessarily $\mathcal{R}_0=\frac{G'(0)c}{ab}>1$) and
\[\lim_{t\to\yy}(u_1(t,x),u_2(t,x))=(u^*_1,u^*_2) ~ ~ {\rm locally ~ uniformly ~ in~} \mathbb{R}.\]

\sk{\rm(ii)}\, \underline{Vanishing:}  $\lim_{t\to\yy}\big(h(t)-g(t)\big)<\yy$ and \[\lim_{t\to\yy}\bigg(\|u_1(t,\cdot)\|_{C([g(t),h(t)])}+\|u_2(t,\cdot)\|_{C([g(t),h(t)])}\bigg)=0.\]

It is easy to show that conditions {\bf(f1)}-{\bf(f5)} hold for $F$. Hence Theorem \ref{t1.1} is valid for \eqref{4.2}.

\begin{theorem}\label{t4.1}Let $(u_1,u_2,g,h)$ be a solution of \eqref{4.2} and spreading happen. Then
\bess\left\{\begin{aligned}
&\lim_{t\to\yy}\max_{x\in[0,\,ct]}\sum_{i=1}^{2}|u_i(t,x)-u^*_i|=0 ~ {\rm for ~ any ~ } c\in(0,c_0) ~ ~ {\rm if ~ {\bf(J1)} ~ holds ~ with ~ }m_0=m=2,\\
&\lim_{t\to\yy}\max_{x\in[0,\,ct]}\sum_{i=1}^{2}|u_i(t,x)-u^*_i|=0 ~ {\rm for ~ any ~ } c>0 ~ ~ {\rm if ~ {\bf(J1)} ~ does ~ not ~ hold},
\end{aligned}\right.\eess
where $c_0$ is uniquely determined by the semi-wave problem \eqref{1.3}-\eqref{1.4} with $m_0=m=2$.
\end{theorem}
However, one easily checks that $F$ does not satisfy {\bf(f6)}. Thus Theorem \ref{t1.2} can not directly be applied to \eqref{4.2}. But by using some new lower solution we still can prove that the similar results in Theorem \ref{t1.2} hold for problem \eqref{4.2}.

\begin{theorem}\label{t2.5} Assume that $J_i$ satisfy ${\bf (J^{\gamma})}$ with $\gamma\in(1,2]$ and $m_0=m=2$. Let spreading happen for \eqref{4.2}. Then
\bess\left\{\begin{aligned}
&\lim_{t\to\yy}\max_{|x|\le s(t)}\sum_{i=1}^{2}|u_i(t,x)-u^*_i|=0  ~ {\rm for ~ any } ~ 0\le s(t)=t^{\frac{1}{\gamma-1}}o(1) ~ ~ {\rm  ~ if ~}\gamma\in(1,2),\\
&\lim_{t\to\yy}\max_{|x|\le s(t)}\sum_{i=1}^{2}|u_i(t,x)-u^*_i|=0  ~ {\rm for ~ any } ~ 0\le s(t)=(t\ln t) o(1)  ~ ~ {\rm  ~ if ~ }\gamma=2.
\end{aligned}\right.\eess
\end{theorem}
\begin{proof}{\bf Step 1:} Consider problem
\bess\left\{\begin{aligned}
&\bar{u}_{1t}=-a\bar{u}_1+c\bar{u}_2\\
&\bar{u}_{2t}=-b\bar{u}_2+G(u_1)\\
&\bar{u}_1(0)=\|u_{10}(x)\|_{C([-h_0,h_0])}, ~ \bar{u}_2(0)=\|u_{20}(x)\|_{C([-h_0,h_0])}.
\end{aligned}\right.\eess
It follows from simple phase-plane analysis that $\lim_{t\to\yy}\bar{u}_1(t)=u^*_1$ and $\lim_{t\to\yy}\bar{u}_2(t)=u^*_2$. By a comparison argument, we have $u_1(t,x)\le \bar{u}_1(t)$ and $u_2(t,x)\le \bar{u}_2(t)$ for $t\ge0$ and $x\in\mathbb{R}$. Thus
\[\limsup_{t\to\yy}u_1(t,x)\le u^*_1 ~ {\rm and } ~ \limsup_{t\to\yy}u_2(t,x)\le u^*_2 ~ ~ {\rm uniformly ~ in ~ }\mathbb{R}.\]
Thus it remains to show the lower limitations of $u$. We will carry it out in two steps.

{\bf Step 2:} This step concerns the case $\gamma\in(1,2)$. We will construct a suitably lower solution, which is different from that of the Step 2 in proof of Theorem \ref{t1.2}, to show that for any $0\le s(t)=t^{\frac{1}{\gamma-1}}o(1)$,
\[\liminf_{t\to\yy}u_1(t,x)\ge u^*_1 ~ {\rm and } ~ \liminf_{t\to\yy}u_2(t,x)\ge u^*_2 ~ ~ {\rm uniformly ~ in ~ }|x|\le s(t).\]

For small $\ep>0$ and $0<\frac{\alpha_2}{2}<\alpha_1<\alpha_2<1$, define
\bess
&&\underline{h}(t)=(Kt+\theta)^{\frac{1}{\gamma-1}} ~ {\rm and}\\ &&\underline{u}(t,x)=(\underline{u}_1(t,x),\underline{u}_2(t,x))=\bigg(u^*_1(1-\ep^{\alpha_1})(1-\frac{|x|}{\underline{h}(t)}),
u^*_2(1-\ep^{\alpha_2})(1-\frac{|x|}{\underline{h}(t)})\bigg)
\eess
with $t\ge0$, $x\in[-\underline{h}(t),\underline{h}(t)]$, and $K,\theta>0$ to be determined later.

We next prove that for small $\ep>0$, there exist proper $T,K$ and $\theta>0$ such that
\bes\left\{\begin{aligned}\label{4.3}
&\partial_t\underline u_1\le d_1\int_{-\underline{h}(t)}^{\underline{h}(t)}J_1(x-y)\underline u_1(t,y)\dy-d_1\underline u_1-a\underline u_1+c\underline u_2, &&\hspace{-15mm}t>0,~x\in(-\underline h(t),\underline h(t)),\\
&\partial_t\underline u_2\le d_2\int_{-\underline{h}(t)}^{\underline{h}(t)}J_2(x-y)\underline u_2(t,y)\dy-d_2\underline u_2-b\underline u_2+G(\underline u_1), & &\hspace{-15mm}t>0,~x\in(-\underline h(t),\underline h(t)),\\
&\underline u_i(t,\pm\underline h(t))\le0,& &\hspace{-15mm}t>0, ~ i=1,2,\\
&-\underline{h}(t)\ge-\sum_{i=1}^{2}\mu_i\int_{-\underline{h}(t)}^{h(t)}\int_{-\yy}^{-h(t)}
J_i(x-y)\underline u_i(t,x)\dy\dx,& &\hspace{-15mm}t>0,\\
&\underline h'(t)\le\sum_{i=1}^{2}\mu_i\int_{-\underline h(t)}^{h(t)}\int_{\underline h(t)}^{\infty}
J_i(x-y)\underline u_i(t,x)\dy\dx,& &\hspace{-15mm}t>0,\\
&-\underline{h}(0)\ge g(T),~ \underline h(0)\le h(T);\;\;\underline u_1(0,x)\le u_1(T,x),\;\;\underline u_2(0,x)\le u_{2}(T,x),&&x\in[-\underline{h}(0),\underline{h}(0)].
 \end{aligned}\right.
 \ees
 As before, if \eqref{4.3} holds, then by comparison principle and our definition of the lower solution $(\underline{u},-\underline{h},\underline{h})$, we easily derive
 \[\liminf_{t\to\yy}(u_1(t,x),u_2(t,x))\succeq(u^*_1(1-\ep^{\alpha_1}),u^*_2(1-\ep^{\alpha_2})) ~ {\rm uniformly ~ in ~} |x|\le s(t),\]
 which, combined with the arbitrariness of $\ep$, yields
 \[\liminf_{t\to\yy}(u_1(t,x),u_2(t,x))\succeq(u^*_1,u^*_2) ~ {\rm uniformly ~ in ~} |x|\le s(t).\]

Clearly, $\underline{u}_i(t,\pm\underline{h}(t))=0$ for $t\ge0$. Then we prove that the fourth and fifth inequalities of \eqref{4.3} hold. Similarly to the proof of Theorem \ref{t1.2}, for large $\theta>0$ and small $K>0$, we have
\bess
&&\sum_{i=1}^{2}\mu_i\dd\int_{-\underline h(t)}^{\underline h(t)}\int_{\underline h(t)}^{\infty}
J_i(x-y)\underline u_i(t,x)\dy\dx\\
&&=\sum_{i=1}^{2}(1-\ep^{\alpha_i})\mu_iu^*_i\dd\int_{-\underline h(t)}^{\underline h(t)}\int_{\underline h(t)}^{\infty}
J_i(x-y)(1-\frac{|x|}{\underline{h}(t)})\dy\dx\\
&&\ge\sum_{i=1}^{2}\frac{(1-\ep^{\alpha_i})\mu_iu^*_i}{\underline{h}(t)}\bigg(\dd\int_{0}^{\underline h(t)}\int_{0}^{y}+\int_{\underline{h}(t)}^{\yy}\int_{0}^{\underline{h}(t)}\bigg)J_i(y)x\dx\dy\\
&&\ge\sum_{i=1}^{2}\frac{(1-\ep^{\alpha_i})\mu_iu^*_i}{2\underline{h}(t)}\int_{0}^{\underline h(t)}J_i(y)y^2\dy\ge\sum_{i=1}^{2}\frac{C_1(1-\ep^{\alpha_i})\mu_iu^*_i}{2\underline{h}(t)}\int_{\underline h(t)/2}^{\underline h(t)}y^{2-\gamma}\dy\\
&&\ge\tilde{C}_1(Kt+\theta)^{(2-\gamma)/(\gamma-1)}\ge\frac{K(Kt+\theta)^{(2-\gamma)/(\gamma-1)}}{\gamma-1}=\underline{h}'(t)
.
\eess
Thus the fourth inequality in \eqref{4.3} holds, and the fifth obviously follows from the symmetry of $J$ and $\underline{u}$ about $x$. Now we deal with the first two inequalities in \eqref{4.3}. As in the Step 2 of the proof of Theorem \ref{t1.2}, we can show that there exist $\hat{C}>0$ such that
\bes\label{4.4}\int_{-\underline h(t)}^{\underline h(t)}J_i(x-y)\underline u_i(t,y)\dy\ge\hat{C}u^*_i(1-\ep^{\alpha_i})\underline h^{1-\gamma}(t) ~ ~ {\rm for } ~ t>0, ~ x\in[-\underline h(t),\underline h(t)].
\ees
So the details are omitted here. Then we claim that for small $\ep>0$,
\bes\label{4.5}-a\underline{u}_1(t,x)+c\underline{u}_2(t,x)\ge\ep\underline{u}_1(t,x) ~ ~ {\rm for ~ }t\ge0, ~ x\in[-\underline h(t),\underline h(t)].\ees
Clearly, it suffices to show that if $\ep$ is small enough, then
\[-au^*_1(1-\ep^{\alpha_1})+cu^*_2(1-\ep^{\alpha_2})\ge\ep u^*_1(1-\ep^{\alpha_1}).\]
Direct computations show
\bess
&&-au^*_1(1-\ep^{\alpha_1})+cu^*_2(1-\ep^{\alpha_2})-\ep u^*_1(1-\ep^{\alpha_1})\\
&&=-au^*_1+au^*_1\ep^{\alpha_1}+cu^*_2-cu^*_2\ep^{\alpha_2}-\ep u^*_1(1-\ep^{\alpha_1})\\
&&=au^*_1\ep^{\alpha_1}-cu^*_2\ep^{\alpha_2}-\ep u^*_1(1-\ep^{\alpha_1})\\
&&=\ep^{\alpha_1}\left[au^*_1-cu^*_2\ep^{\alpha_2-\alpha_1}-\ep^{1-\alpha_1} u^*_1(1-\ep^{\alpha_1})\right]>0 ~ ~ {\rm if ~ } \ep ~ {\rm is ~ small ~ sufficiently}.
\eess
Furthermore, we claim that for small $\ep>0$,
\bes\label{4.6}-b\underline{u}_2(t,x)+G(\underline{u}_1(t,x))\ge\ep\underline{u}_2(t,x)~ ~ {\rm for ~ }t\ge0, ~ x\in[-\underline h(t),\underline h(t)].
\ees
We first prove that for small $\ep>0$,
\bes\label{4.7}
\frac{G(u^*_1(1-\ep^{\alpha_1})(1-\frac{|x|}{\underline{h}(t)}))}
{u^*_1(1-\ep^{\alpha_1})(1-\frac{|x|}{\underline{h}(t)})}\ge\frac{ab}{c}(1+\ep^{\alpha_1})~ ~ {\rm for ~ }t\ge0, ~ x\in(-\underline h(t),\underline h(t)).
\ees
By the assumptions on $G$, we see
\[\frac{G(u^*_1(1-\ep^{\alpha_1})(1-\frac{|x|}{\underline{h}(t)}))}
{u^*_1(1-\ep^{\alpha_1})(1-\frac{|x|}{\underline{h}(t)})}\ge\frac{G(u^*_1(1-\ep^{\alpha_1}))}
{u^*_1(1-\ep^{\alpha_1})}.\]
Thus we only need to prove that for $t\ge0$ and $x\in(-\underline h(t),\underline h(t))$,
\[\frac{G(u^*_1(1-\ep^{\alpha_1}))}
{u^*_1(1-\ep^{\alpha_1})}\ge\frac{ab}{c}(1+\ep^{\alpha_1}).\]
Define
\[\Gamma(\ep)=\frac{G(u^*_1(1-\ep^{\alpha_1}))}
{u^*_1(1-\ep^{\alpha_1})}-\frac{ab}{c}(1+\ep^{\alpha_1}) ~ ~ {\rm for ~ }0<\ep\ll1.\]
Obviously, $\Gamma(0)=0$, and by our assumptions on $G$ we obtain that for $0<\ep\ll1$,
\bess
&&\Gamma'(\ep)=-\bigg(\frac{G(u^*_1(1-\ep^{\alpha_1}))}
{u^*_1(1-\ep^{\alpha_1})}\bigg)'\alpha_1\ep^{\alpha_1-1}-\frac{ab}{c}\alpha_1\ep^{\alpha_1-1}\\
&&=\alpha_1\ep^{\alpha_1-1}\left[-\bigg(\frac{G(u^*_1(1-\ep^{\alpha_1}))}
{u^*_1(1-\ep^{\alpha_1})}\bigg)'-\frac{ab}{c}\right]>0.
\eess
So \eqref{4.7} holds. Now we continue to prove \eqref{4.6}. Obviously, it holds for $x=\pm\underline{h}(t)$. For $x\in(-\underline h(t),\underline h(t))$, we have
\bess
&&-b\underline{u}_2(t,x)+G(\underline{u}_1(t,x))-\ep\underline{u}_2(t,x)\\
&&=(1-\frac{|x|}{\underline{h}(t)})\left[-bu^*_2(1-\ep^{\alpha_2})+u^*_1(1-\ep^{\alpha_1})
\frac{G(u^*_1(1-\ep^{\alpha_1})(1-\frac{|x|}{\underline{h}(t)}))}
{u^*_1(1-\ep^{\alpha_1})(1-\frac{|x|}{\underline{h}(t)})}-\ep u^*_2(1-\ep^{\alpha_2})\right]\\
&&\ge(1-\frac{|x|}{\underline{h}(t)})\left[-bu^*_2(1-\ep^{\alpha_2})+u^*_1(1-\ep^{\alpha_1})
\frac{ab}{c}(1+\ep^{\alpha_1})-\ep u^*_2(1-\ep^{\alpha_2})\right]\\
&&=(1-\frac{|x|}{\underline{h}(t)})\left[bu^*_2\ep^{\alpha_2}-
\frac{abu^*_1}{c}\ep^{2\alpha_1}-\ep u^*_2(1-\ep^{\alpha_2})\right]\\
&&=(1-\frac{|x|}{\underline{h}(t)})\ep^{\alpha_2}\left[bu^*_2-
\frac{abu^*_1}{c}\ep^{2\alpha_1-\alpha_2}-\ep^{1-\alpha_2} u^*_2(1-\ep^{\alpha_2})\right]>0
\eess
provided that $\ep$ is small suitably. With \eqref{4.4}, \eqref{4.5} and \eqref{4.6} in hand, we similarly can obtain that for small $K>0$,
\bess
&&d_1\int_{-\underline{h}(t)}^{\underline{h}(t)}J_1(x-y)\underline u_1(t,y)\dy-d_1\underline u_1-a\underline u_1+c\underline u_2\\
&&\ge\frac{\ep}{2}\int_{-\underline{h}(t)}^{\underline{h}(t)}J_1(x-y)\underline u_1(t,y)\dy\ge\hat{C}u^*_1(1-\ep^{\alpha_1})\underline h^{1-\gamma}(t)\ge\frac{u^*_1(1-\ep^{\alpha_1})K\underline{h}^{1-\gamma}}{\gamma-1}\ge\underline{u}_{1t},
\eess
and
\bess
&&d_2\int_{-\underline{h}(t)}^{\underline{h}(t)}J_2(x-y)\underline u_2(t,y)\dy-d_2\underline u_2-b\underline u_2+G(\underline u_1)\\
&&\ge\frac{\ep}{2}\int_{-\underline{h}(t)}^{\underline{h}(t)}J_2(x-y)\underline u_2(t,y)\dy
\ge\hat{C}u^*_2(1-\ep^{\alpha_2})\underline h^{1-\gamma}(t)\ge\frac{u^*_2(1-\ep^{\alpha_2})K\underline{h}^{1-\gamma}}{\gamma-1}\ge\underline{u}_{2t}.
\eess
Therefore the first two inequalities in \eqref{4.3} hold.

Since spreading happens, for such $K,\theta$ and $\ep$ as chosen above, there is a $T>0$ such that $-\underline{h}(0)\ge g(T)$, $\underline{h}(0)\le h(T)$ and $\underline{u}(0,x)\preceq (u^*_1(1-\ep^{\alpha_1}),u^*_2(1-\ep^{\alpha_2}))\preceq(u_1(T,x),u_2(T,x))$ in $[-\underline{h}(0),\underline{h}(0)]$. Hence \eqref{4.3} hold, and we finish this step.

{\bf Step 3:} We now prove that for any $0\le s(t)=t\ln t o(1)$,
\[\liminf_{t\to\yy}u_1(t,x)\ge u^*_1 ~ {\rm and } ~ \liminf_{t\to\yy}u_2(t,x)\ge u^*_2 ~ ~ {\rm uniformly ~ in ~ }|x|\le s(t).\]
For fixed $0<\frac{\alpha_2}{2}<\alpha_1<\alpha_2<1$ and small $\ep>0$, define
\bess
&&\underline{h}(t)=K(t+\theta)\ln(t+\theta) ~ {\rm and}\\ &&\underline{u}(t,x)=(\underline{u}_1(t,x),\underline{u}_2(t,x))
=\bigg(u^*_1(1-\ep^{\alpha_1})\min\{1,\frac{\underline{h}(t)-|x|}{(t+\theta)^{1/2}}\},
u^*_2(1-\ep^{\alpha_2})\min\{1,\frac{\underline{h}(t)-|x|}{(t+\theta)^{1/2}}\}\bigg)
\eess
with $t\ge0$, $x\in[-\underline{h}(t),\underline{h}(t)]$, and $K,\theta>0$ to be determined later.

We next prove that for small $\ep>0$, there exist proper $T,K$ and $\theta>0$ such that
\bes\left\{\begin{aligned}\label{4.8}
&\partial_t\underline u_1\le d_1\int_{-\underline{h}(t)}^{\underline{h}(t)}J_1(x-y)\underline u_1(t,y)\dy-d_1\underline u_1-a\underline u_1+c\underline u_2,\\
 &\hspace{25mm}t>0,~x\in(-\underline h(t),\underline h(t))\setminus\{\underline{h}(t)-(t+\theta)^{1/2}\},\\
&\partial_t\underline u_2\le d_2\int_{-\underline{h}(t)}^{\underline{h}(t)}J_2(x-y)\underline u_2(t,y)\dy-d_2\underline u_2-b\underline u_2+G(\underline u_1),\\
&\hspace{25mm}t>0,~x\in(-\underline h(t),\underline h(t))\setminus\{\underline{h}(t)-(t+\theta)^{1/2}\},\\
&\underline u_i(t,\pm\underline h(t))\le0,& &\hspace{-25mm}t>0, ~ i=1,2,\\
&-\underline{h}(t)\ge-\sum_{i=1}^{2}\mu_i\int_{-\underline{h}(t)}^{h(t)}\int_{-\yy}^{-h(t)}
J_i(x-y)\underline u_i(t,x)\dy\dx,& &\hspace{-25mm}t>0,\\
&\underline h'(t)\le\sum_{i=1}^{2}\mu_i\int_{-\underline h(t)}^{h(t)}\int_{\underline h(t)}^{\infty}
J_i(x-y)\underline u_i(t,x)\dy\dx,& &\hspace{-25mm}t>0,\\
&-\underline{h}(0)\ge g(T),~ \underline h(0)\le h(T);\;\;\underline u_1(0,x)\le u_1(T,x),\;\;\underline u_2(0,x)\le u_{2}(T,x),&&x\in[-\underline{h}(0),\underline{h}(0)].
 \end{aligned}\right.
 \ees
Once \eqref{4.8} is derived, we similarly can complete the step.  It is not hard to verify that \eqref{4.5} and \eqref{4.6} are still valid for small $\ep>0$. Then by following similar lines with the proof of Theorem \ref{t1.2}, we immediately verify \eqref{4.8}. The details are omitted. The proof is complete.
\end{proof}
On the other hand, noticing that the growth rate of infectious agents is of concave nonlinearity, Hsu and Yang \cite{HY} recently proposed the following variation of model \eqref{4.1}
\bes\label{4.a}
\left\{\begin{aligned}
&\partial_tu_1=d_1\Delta u_1-au_1+H(u_2), & &t>0,~x\in\Omega,\\
&\partial_tu_2=d_2\Delta u_2-bu_2+G(u_1), & &t>0,~x\in\Omega,
\end{aligned}\right.
 \ees
 where $H(u_2)$ and $G(u_1)$ satisfy that $H,G\in C^2([0,\yy))$, $H(0)=G(0)=0$, $H',G'>0$ in $[0,\yy)$, $H^{''},G^{''}>0$ in $(0,\yy)$, and $G(H(\hat z)/a)<b\hat{z}$ for some $\hat{z}$.
Examples for such $H$ and $G$ are $H(z)=\alpha z/(1+z)$ and $G(z)=\beta \ln(z+1)$ with $\alpha,\beta>0$ and $\alpha\beta>ab$. Based on the above assumptions, it is easy to show that if $0<H'(0)G'(0)/(ab)\le1$, the unique nonnegative constant equilibrium is $(0,0)$, and if $H'(0)G'(0)/(ab)>1$, there are only two nonnegative constant equilibria, i.e., $(0,0)$ and $(u^*_1,u^*_2)\succ{\bf0}$.
Some further results about \eqref{4.a} can be seen from \cite{HY} and \cite{WH}. 

Motivated by the above works, Nguyen and Vo \cite{NV} very recently incorporated nonlocal diffusion and free boundary into model \eqref{4.a}, and thus obtained the following problem
\bes\left\{\begin{aligned}\label{4.b}
&\partial_tu_1=d_1\int_{g(t)}^{h(t)}J_1(x-y)u_1(t,y)\dy-d_1u_1-au_1+H(u_2), & &t>0,~x\in(g(t),h(t)),\\
&\partial_tu_2=d_2\int_{g(t)}^{h(t)}J_2(x-y)u_2(t,y)\dy-d_2u_2-bu_2+G(u_1), & &t>0,~x\in(g(t),h(t)),\\
&u_i(t,g(t))=u_i(t,h(t))=0,& &t>0, ~ i=1,2,\\
&g'(t)=-\sum_{i=1}^{2}\mu_i\int_{g(t)}^{h(t)}\int_{-\yy}^{g(t)}
J_i(x-y)u_i(t,x)\dy\dx,& &t>0,\\
&h'(t)=\sum_{i=1}^{2}\mu_i\int_{g(t)}^{h(t)}\int_{h(t)}^{\infty}
J_i(x-y)u_i(t,x)\dy\dx,& &t>0,\\
&-g(0)=h(0)=h_0>0;\;\;u_1(0,x)=u_{10}(x),\;\;u_2(0,x)=u_{20}(x),&&x\in[-h_0,h_0].
 \end{aligned}\right.
 \ees
 They proved that problem \eqref{4.b} has a unique global solution, and its dynamics are also governed by a spreading-vanishing dichotomy. Now we give more accurate estimates on longtime behaviors of the solution to \eqref{4.b}. Assume $H'(0)G'(0)/(ab)>1$ and define
 \[F(u_1,u_2)=(f_1(u_1,u_2),f_2(u_1,u_2))=(-au_1+H(u_2),-bu_2+G(u_1)) ~ ~ {\rm and } ~ ~ {\bf u^*}=(u^*_1,u^*_2).\]
 One can easily check that {\bf (f1)}-{\bf(f6)} hold with ${\bf \hat{u}}={\bf\yy}$. Thus Theorem \ref{t1.1} and Theorem \ref{t1.2} are valid for solution of problem \eqref{4.b}. Here for convenience of readers, we list the results as below.
 \begin{theorem}\label{t4.a}Let $(u_1,u_2,g,h)$ be a solution of \eqref{4.b} and $m_0=m=2$ in conditions {\bf (J1)} and ${\bf(J^\gamma)}$. If spreading happens, then
 \bess\left\{\begin{aligned}
&\hspace{-2mm}\lim_{t\to\yy}\max_{|x|\le ct}[|u_1(t,x)-u^*_1|+|u_2(t,x)-u^*_2|]=0 ~ {\rm for ~ any ~ } c\in(0,c_0) ~ ~ {\rm if ~ {\bf(J1)} ~ holds},\\
&\hspace{-2mm}\lim_{t\to\yy}\max_{|x|\le ct}[|u_1(t,x)-u^*_1|+|u_2(t,x)-u^*_2|]=0 ~ {\rm for ~ any ~ } c>0 ~ ~ {\rm if ~ {\bf(J1)} ~ does ~ not ~ hold},\\
&\hspace{-2mm}\lim_{t\to\yy}\max_{|x|\le s(t)}[|u_1(t,x)-u^*_1|+|u_2(t,x)-u^*_2|]=0  ~ {\rm for ~ any } ~s(t)=t^{\frac{1}{\gamma-1}}o(1) \; {\rm if ~ {\bf(J^\gamma)} ~ holds ~ for } ~\gamma\in(1,2),\\
&\hspace{-2mm}\lim_{t\to\yy}\max_{|x|\le s(t)}[|u_1(t,x)-u^*_1|+|u_2(t,x)-u^*_2|]=0 \; {\rm for ~ any \; } s(t)=(t\ln t) o(1) \; {\rm if ~ {\bf(J^\gamma)} ~ holds ~ for } ~\gamma=2.
\end{aligned}\right.\eess
where $c_0$ is uniquely determined by the semi-wave problem \eqref{1.3}-\eqref{1.4} with $m_0=m=2$.
 \end{theorem}
 
{\bf Example 2.} Our second example is the following West Nile virus model with nonlocal diffusion and free boundaries
\bes\left\{\begin{aligned}\label{4.9}
&H_t=d_1\int_{g(t)}^{h(t)}J_1(x-y)H(t,y)\dy-d_1H+a_1(e_1-H)V-b_1H, & &t>0,~x\in(g(t),h(t)),\\
&V_t=d_2\int_{g(t)}^{h(t)}J_2(x-y)V(t,y)\dy-d_2V+a_2(e_2-V)H-b_2V, & &t>0,~x\in(g(t),h(t)),\\
&H(t,x)=V(t,x)=0,& &t>0, ~ x\in\{g(t),h(t)\},\\
&g'(t)=-\mu\int_{g(t)}^{h(t)}\int_{-\yy}^{h(t)}
J_2(x-y)H(t,x)\dy\dx,& &t>0,\\
&h'(t)=\mu\int_{g(t)}^{h(t)}\int_{h(t)}^{\infty}
J_2(x-y)H(t,x)\dy\dx,& &t>0,\\
&-g(0)=h(0)=h_0>0;\;\;H(0,x)=u_{10}(x),\;\;V(0,x)=u_{20}(x),&&x\in[-h_0,h_0],
 \end{aligned}\right.
 \ees
 where $J_i$ satisfy {\bf(J)}. $d_i,a_i,b_i,e_i$ and $\mu$ are positive constant. $H(t,x)$ and $V(t,x)$ are the densities of the infected bird (host) and mosquito (vector)
populations, respectively. Since model \eqref{4.9} is a simplified model, the biological interpretation of this model can be referred to the literatures \cite{ALZ,LRD,LH,WHD}. Moreover, the dynamics of this model have been recently studied in \cite{DNwn}. The authors proved that the dynamics of \eqref{4.9} are governed by the spreading-vanishing dichotomy

  \sk{\rm(i)}\, \underline{Spreading:} $\lim_{t\to\yy}-g(t)=\lim_{t\to\yy}h(t)=\yy$ (necessarily $\mathcal{R}_0=\sqrt\frac{a_1a_2e_1e_2}{b_1b_2}>1$) and
\[\lim_{t\to\yy}(H(t,x),V(t,x))=(\frac{a_1a_2e_1e_2-b_1b_2}{a_1a_2e_2+b_1a_2},\frac{a_1a_2e_1e_2-b_1b_2}{a_1a_2e_1+a_1b_2}) ~ ~ {\rm locally ~ uniformly ~ in~} \mathbb{R}.\]

\sk{\rm(ii)}\, \underline{Vanishing:}  $\lim_{t\to\yy}\big(h(t)-g(t)\big)<\yy$ and \[\lim_{t\to\yy}\bigg(\|H(t,\cdot)\|_{C([g(t),h(t)])}+\|V(t,\cdot)\|_{C([g(t),h(t)])}\bigg)=0.\]

  Assume that spreading happens (necessarily $a_1a_2e_1e_2>b_1b_2$). Thus we have
 \bess
 &F(u)=(f_1(u_1,u_2),f_2(u_1,u_2))=((a_1(e_1-u_1)u_2-b_1u_1,a_2(e_2-u_2)u_1-b_2u_2)\\
&{\bf u^*}=(u^*_1,u^*_2)=(\frac{a_1a_2e_1e_2-b_1b_2}{a_1a_2e_2+b_1a_2},\frac{a_1a_2e_1e_2-b_1b_2}{a_1a_2e_1+a_1b_2}).
\eess
 One easily checks that conditions ${\bf (f_1)}-{\bf (f_6)}$ hold for $F$ with ${\bf \hat{u}}=(e_1,e_2)$. Hence the more accurate longtime behaviors of solution to \eqref{4.9} can be summarized as follows.
 \begin{theorem}\label{t4.2}Let $(H,V,g,h)$ be a solution of \eqref{4.9} and $m_0=m=2$ in conditions {\bf (J1)} and ${\bf(J^\gamma)}$. If spreading happens, then
 \bess\left\{\begin{aligned}
&\hspace{-2mm}\lim_{t\to\yy}\max_{|x|\le ct}[|H(t,x)-u^*_1|+|V(t,x)-u^*_2|]=0 ~ {\rm for ~ any ~ } c\in(0,c_0) ~ ~ {\rm if ~ {\bf(J1)} ~ holds},\\
&\hspace{-2mm}\lim_{t\to\yy}\max_{|x|\le ct}[|H(t,x)-u^*_1|+|V(t,x)-u^*_2|]=0 ~ {\rm for ~ any ~ } c>0 ~ ~ {\rm if ~ {\bf(J1)} ~ does ~ not ~ hold},\\
&\hspace{-2mm}\lim_{t\to\yy}\max_{|x|\le s(t)}[|H(t,x)-u^*_1|+|V(t,x)-u^*_2|]=0  ~ {\rm for ~ any } ~s(t)=t^{\frac{1}{\gamma-1}}o(1) \; {\rm if ~ {\bf(J^\gamma)} ~ holds ~ for } ~\gamma\in(1,2),\\
&\hspace{-2mm}\lim_{t\to\yy}\max_{|x|\le s(t)}[|H(t,x)-u^*_1|+|V(t,x)-u^*_2|]=0 \; {\rm for ~ any \; } s(t)=(t\ln t) o(1) \; {\rm if ~ {\bf(J^\gamma)} ~ holds ~ for } ~\gamma=2.
\end{aligned}\right.\eess
where $c_0$ is uniquely determined by the semi-wave problem \eqref{1.3}-\eqref{1.4} with $m_0=m=2$.
 \end{theorem}

\end{document}